\documentclass
[12pt, reqno]{amsart}
\usepackage{amsmath,amsthm,amssymb,mathrsfs,color}
\usepackage{mathrsfs}
\usepackage[varg]{txfonts}
\newtheorem{tm}{Theorem}[section]
\newtheorem{lm}[tm]{Lemma}
\newtheorem{co}[tm]{Corollary}
\newtheorem{re}[tm]{Remark}

\newtheorem{pr}[tm]{Proposition}

\allowdisplaybreaks[4]
\makeatletter
\newcommand{\subscripts}[3]{%
  \@mathmeasure\z@\displaystyle{#2}%
  \global\setbox\@ne\vbox to\ht\z@{}\dp\@ne\dp\z@
  \setbox\tw@\box\@ne
  \@mathmeasure4\displaystyle{\copy\tw@_{#1}}%
  \@mathmeasure6\displaystyle{{#2}_{#3}}%
  \dimen@-\wd6 \advance\dimen@\wd4 \advance\dimen@\wd\z@
  \hbox to\dimen@{}\mathop{\kern-\dimen@\box4\box6}%
}
\makeatother
\makeatletter
 \renewcommand{\theequation}{%
      \thesection.\arabic{equation}}
\@addtoreset{equation}{section}
\makeatother

\begin{document}
\title[Strong uniqueness of Dirichlet operators related to stochastic quantization]
{Strong uniqueness for Dirichlet operators related to stochastic quantization 
under exponential/trigonometric interactions 
on the two-dimensional torus}
\author[S. Albeverio]{Sergio Albeverio}
\address[Sergio Albeverio]{Institut f\"ur Angewandte Mathematik and Hausdorff Center for Mathematics,
Universit\"at Bonn, Endenicher Allee 60, D-53115 Bonn, Germany} 
\email{albeverio@iam.uni-bonn.de}
\author[H. Kawabi]{Hiroshi Kawabi}
\address[Hiroshi Kawabi]{Department of Mathematics, Hiyoshi Campus, Keio University, 
4-1-1, Hiyoshi, Kohoku-ku, Yokohama 223-8521, Japan}
\email{kawabi@keio.jp}
\author[S.-R. Mihalache]{Stefan-Radu Mihalache}
\address[Stefan-Radu Mihalache]{is employed at a German federal authority.
}
\email{stefanmihalache@gmx.de}
\author[M. R\"ockner]{Michael R\"ockner}
\address[Michael R\"ockner]{Fakult\"at f\"ur Mathematik,
Universit\"at Bielefeld,
Universit\"atsstra{\ss}e 25, 
D-33501 Bielefeld, Germany, and
Academy of Mathematics and Systems Science, 
Chinese Academy of Sciences, 
55, Zhongguancun East Road, 
Beijing, 100190, China}
\email{roeckner@math.uni-bielefeld.de}

\maketitle
\begin{abstract}
We consider space-time quantum fields with
exponential/trigonometric interactions. 
In the context of Euclidean quantum field theory, the former and the latter are called the H{\o}egh-Krohn model 
and the Sine-Gordon model, respectively.
The main objective of the present paper is to construct infinite dimensional diffusion processes 
which solve modified stochastic quantization equations for these quantum fields on
the two-dimensional torus
by the Dirichlet form approach and to
prove strong uniqueness of
the corresponding Dirichlet operators.
\end{abstract}
\section{Introduction}
In recent years, there has been a growing interest in 
the 
study of 
infinite dimensional stochastic dynamics
associated with various models of
Euclidean quantum field theory, hydrodynamics
and statistical mechanics, see, e.g., 
\cite{AFS, 
AKKR-book, H, GIP, 
AKus-19} and references therein.
One of the principal programs
in these studies is to obtain Gibbs measures 
as equilibrium states of
stochastic processes taking values in infinite dimensional state spaces.

In Euclidean quantum field theory, the space-time free field in finite volume 
is given by the Gaussian measure $\mu_{0}$ on 
a Sobolev space of negative order
$H^{-\delta}({\mathbb T}^{2})$ defined on the 2-dimensional torus ${\mathbb T}^{2}=
(\mathbb R/ 2\pi \mathbb Z)^{2}$
with zero mean and covariance operator $(1-\Delta)^{-1}$, $\Delta$ being
the Laplacian in $L^{2}(\mathbb T^{2})$ with periodic boundary conditions.
Heuristically,
$\mu_{0}$ is given by the expression
\begin{equation}
\mu_{0}(d\phi)
\varpropto
\exp \Big \{ -\frac{1}{2} 
\int_{{\mathbb T}^{2}} \Big( \phi(x)^{2}+\vert \nabla \phi(x) \vert_{{\mathbb R}^{2}}^{2} \Big) \hspace{0.5mm}
dx
\Big \}
\prod_{x\in {\mathbb T}^{2}} d\phi(x),
\label{formal-free}
\nonumber
\end{equation}
and it is also called the {\it{massive Gaussian free field}}.
The $\phi^{2m}_{2}$-{\it{quantum field}}, a special case of the
$P(\phi)_{2}$-{\it{quantum field in finite volume}}, is one of the most important objects
of study in Euclidean quantum field theory (see e.g., \cite{Si, GJ}), and is given by the 
probability measure
\begin{equation}
\mu^{(2m)}_{\sf pol}(d\phi)=\frac{1}{{Z}^{(2m)}_{\sf pol}}
\exp \Big (
- \int_{\mathbb T^{2}}
:\hspace{-0.8mm} \phi^{2m}(x) \hspace{-0.8mm}:
dx \Big ) \hspace{0.5mm}
\mu_{0}(d\phi), \quad m=2,3,4\ldots,
\label{phi42-Gibbs}
\nonumber
\end{equation}
where 
$:\hspace{-0.8mm} \phi^{2m} \hspace{-0.8mm}:$ denotes the $2m$-th order
Wick power of $\phi$ with respect to $\mu_{0}$ and 
${Z}^{(2m)}_{\sf pol}>0$ is the normalizing constant given by
\begin{equation}
{Z}^{(2m)}_{\sf pol}
:=\int_{H^{-\delta}(\mathbb T^{2})} \exp 
\Big(-\int_{\mathbb T^{2}}  :\hspace{-0.8mm} \phi^{2m}(x) \hspace{-0.8mm}:
dx \Big) \hspace{0.5mm} \mu_{0}(d\phi).
\nonumber
\end{equation}

Parisi and Wu \cite{PW} proposed the above program to find a Markov process or a class of Markov
processes so that $\mu^{(2m)}_{\sf pol}$ is their invariant measure. 
This program is now called the {\it{stochastic quantization}}.
(More recently, following 
\cite{H}, such processes with $\mu^{(2m)}_{\sf pol}$ as an invariant measure 
have been called {\it{dynamical $\phi_{2}^{2m}$-models}}.)
In the well-known paper \cite{JM}, 
Jona-Lasinio and Mitter firstly provided a 
rigorous justification of the program for $P(\phi)_{2}$-quantum fields in finite volume.
Strictly speaking, they constructed an infinite dimensional continuous Markov process 
having the $\phi^{4}_{2}$-measure $\mu^{(4)}_{\sf pol}$ (i.e., $P(\phi)_{2}$-measure in the case $m=2$)
as an invariant measure.
Applying the Girsanov transform method to the Ornstein--Uhlenbeck process 
whose unique invariant measure is the space-time free field measure $\mu_{0}$,
they obtained the resulting Markov process as the unique weak solution of the so-called
{\it{(regularized) stochastic quantization equation}} (SQE in short)
\begin{equation}
\partial_{t}X_{t}(x)=-\frac{1}{2}(1-\Delta)^{1-\gamma} X_{t}(x) 
-m(1-\Delta)^{-\gamma}\hspace{-0.5mm}
:\hspace{-0.5mm}X^{2m-1}_{t}(x) \hspace{-0.5mm}:
+(1-\Delta)^{-\frac{\gamma}{2}}
\xi_{t}(x),
\quad t>0,~x\in \mathbb T^{2},
\label{SQE-0}
\end{equation}
where the positive constant $\gamma$
provides a 
regularizing effect on the nonlinear drift
term and on the noise term, $\xi=\{ \xi_{t}=(\xi_{t}(x))_{x\in {\mathbb T}^{2}} \}_{t\geq 0}$ 
being an $\mathbb R$-valued Gaussian 
space-time white noise, that is, 
$\xi$ is
the time derivative of 
a standard $L^{2}(\mathbb T^{2})$-cylindrical Brownian motion 
$\{ B_{t}=(B_{t}(x) )_{x\in {\mathbb T}^{2}}\}_{t\geq 0}$.

Since then, there has been a large number of follow-up papers on stochastic quantization, and 
both theories of SPDEs and Dirichlet forms on infinite
dimensional state spaces have been developed intensively. 
The Dirichlet form approach was later elaborated in a series of papers 
\cite{AR89, AR90, AR-SDE}. 
In particular, in \cite{AR90} the closability result was given, and in \cite{AR89} 
the existence of a diffusion (i.e., a path continuous strong Markov
process) was proved and in \cite{AR-SDE} it was shown that
this process indeed solves
equation (\ref{SQE-0}) in the weak sense for quasi-every initial condition
has been shown for $\gamma \geq 0$. In this way the generator of the diffusion process has been
identified to be the self-adjoint negative operator $\mathcal L$ in $L^{2}(\mu^{(2m)}_{\sf pol})$ 
associated with the classical
quasi regular Dirichlet form given by the $\phi^{2m}_{2}$-measure $\mu^{(2m)}_{\sf pol}$.
The problem of its being uniquely determined by the closure from a minimal domain 
${\mathfrak{F}}C^{\infty}_{b}$ (of cylindrical functions) was solved in the sense of Markov uniqueness by
\cite{RZ} for $\gamma>0$.
Strong uniqueness in the strong sense of essential self-adjointness of ${\mathcal L}$ in the relevant $L^{2}$-space
was solved by 
\cite{LR, DT00}, 
still for $\gamma>0$. 
A (weaker) property of {\it{restricted}} Markov uniqueness has been recently obtained by 
\cite{RZZ} for 
$\gamma=0$. This yields uniqueness within the class of {\it{quasi-regular}} Dirichlet forms whose generators 
extend $({\mathcal L}, {\mathfrak{F}}C^{\infty}_{b})$. Moreover these authors show that the martingale problem for
$({\mathcal L}, {\mathfrak{F}}C^{\infty}_{b})$, up to a certain ``$\mu$-equivalence" is unique, in the sense that there
is a unique strong Markov process solving that martingale problem (cf. \cite{MiRo, AMR}). In this way
also the uniqueness of probabilistic weak solutions of equation (\ref{SQE-0}) with $\gamma=0$ has been shown.
In \cite{RZZ} in addition to \cite{AR-SDE} also results of \cite{DD} are used, which yields strong solutions.
For uniqueness of the invariant measures in terms of density under translations, see \cite{RZZ2}.  

Let us mention at this point that the original motivation
to study measures like the $P(\phi)_{2}$-measure came from 
quantum field theory on Euclidean space-time $\mathbb R^{d}$,
in particular where the 2-dimensional torus $\mathbb T^{2}$ is replaced by
the 2-dimensional space-time $\mathbb R^{2}$. Such measures are called
(infinite volume) $P(\phi)_{2}$-measures and can indeed be constructed from 
limits of measures constructed relative to the torus, as shown in the work \cite{GRS},
see \cite{GJ, Si} for other constructions of such infinite volume measures. 
The invariance of these measures under the action of the Euclidean group is
important to construct the associated relativistic quantum fields, the original
physical motivation for studying Euclidean measures. See e.g., \cite{GJ, Si, Dimock,  
AHkFeL, A-Streit, Summers, AHPRS, HKPS}. The corresponding SQE (\ref{SQE-0})
with $\gamma>0$ has been studied in the sense of associating a Dirichlet form and a
Markov semigroup in 
\cite{BCM}. The construction of an 
associated diffusion has been first performed also for 
$\gamma \geq 0$ in \cite{AR89} and it was shown in \cite{AR-SDE}
that it solves (\ref{SQE-0}) in the infinite volume case.
For partial results on the uniqueness problems (corresponding to those we mentioned in the case of 
the torus model), see \cite{ARZ, LR, RZZ}. The result in \cite{RZZ} also applies to the infinite volume case.
Here recent results of \cite{MW}, related to methods developed in \cite{H, GIP, CC, GH19, AKus-19} are used.
In \cite{AKonR} ergodicity of the stochastic dynamics associated to the infinite volume $P(\phi)_{2}$-models
has been proven in the situation of ``pure phases" (in the sense of \cite{GRS-75}).

In the present paper, we initiate the study of the corresponding SQE for a model different from
the $P(\phi)_{2}$-case, but which also leads to interesting relativistic quantum fields, as shown at about
the same time as for the $P(\phi)_{2}$-case. This quantum field model was introduced by H\o egh-Krohn
\cite{Hk-71} in a Hamiltonian setting, and its Euclidean version was constructed  in \cite{AHk}.
In the latter paper \cite{AHk}, the probability measure
\begin{equation}
\mu^{(a)}_{\sf exp}(d\phi)=\frac{1}{Z^{(a)}_{\sf exp}}
\exp \Big (
- \int_{\mathbb T^{2}}
:\hspace{-0.8mm} \exp (a\phi)(x) \hspace{-0.8mm}:
dx \Big )
\hspace{0.5mm}
\mu_{0}(d\phi)
\nonumber
\label{exp-Gibbs}
\end{equation}
was constructed and shown
to yield interesting relativistic quantum fields, where
$Z^{(a)}_{\sf exp}>0$ is the normalizing constant, $a\in 
( -{\sqrt{4\pi}}, {\sqrt{4\pi}})$ is called the {\it{charge parameter}}, and the {\it{Wick exponential}} 
$:\hspace{-0.8mm} \exp (a\phi)(x) \hspace{-0.8mm}:$
is formally introduced by the expression
$$ :\hspace{-0.8mm} \exp (a\phi)(x) \hspace{-0.8mm}:
\hspace{0.5mm}
=\exp \Big( a \phi(x)-\frac{a^{2}}{2} {\mathbb E}^{\mu_{0}} [ \phi(x)^{2}] \Big), \qquad x\in {\mathbb T}^{2},
$$
${\mathbb E}^{\mu_{0}}$ standing for the expectation with respect to the massive Gaussian free field $\mu_{0}$
(given heuristically by (\ref{formal-free})).
This model is called the {\it{H\o egh-Krohn model}} or
the {\it{${\rm{exp}}(\phi)_{2}$-model}}.
The papers 
\cite{AGHk,
AHk-martingale} study extensions of the model to $\mathbb R^{d}$ for $d\geq 2$,
and also show relations with martingale theory and multiplicative noise. 
We also mention that
this model
has also been discussed in the framework of
nonstandard analysis (cf. \cite{AHkFeL}) and white noise analysis (cf. \cite{AHPRS, HKPS}),
as well as in connection with relativistic quantum fields on curved spaces \cite{FHN75}.
A connection with 
problems in representation theory of groups of mappings has been discussed in \cite{AHkT},
see also \cite{AHMTT93}, and for new developments, \cite{ADGV16}.
The relevance of this model has also been pointed out
in connection with string theory \cite{AJPS, GT13}, and recently rediscovered in connection 
with topics like Liouville quantum gravity, where
the random measure, called the
{\it{Gaussian multiplicative chaos}},
$$ {\mathcal M}^{(a)}_{\phi}(dx)= \hspace{0.5mm}:\hspace{-0.8mm} \exp (a\phi)(x) \hspace{-0.8mm}: dx, \quad
x\in {\mathbb T}^{2},
\phi \in H^{-\delta}
({\mathbb T}^{2})$$
plays a central role 
(cf. e.g., \cite{Kahane, Kusuoka, DS, RV, AK16, BSS14, DS19, Garban}). 

The H\o egh-Krohn model is also closely connected with the {\it{Sine-Gordon model}} 
(or the {\it{${\rm{cos}}(\phi)_{2}$-quantum field model}})
\begin{equation}
\mu^{(a)}_{\sf cos}(d\phi)=\frac{1}{Z^{(a)}_{\sf cos}}
\exp \Big (
- \int_{\mathbb T^{2}}
:\hspace{-0.8mm} \cos (a\phi)(x) \hspace{-0.8mm}:
dx \Big )
\hspace{0.5mm}
\mu_{0}(d\phi),
\label{cos-Gibbs}
\nonumber
\end{equation}
where 
$:\hspace{-0.8mm} \cos (a\phi)(x) \hspace{-0.8mm}:$ is defined by
$$ :\hspace{-0.8mm} \cos (a\phi)(x) \hspace{-0.8mm}: \hspace{1mm}=\frac{1}{2} \Big (
\hspace{-0.8mm}:\hspace{-0.8mm} \exp ({\sqrt{-1}}a\phi)(x) \hspace{-0.8mm}: 
+
:\hspace{-0.8mm} \exp (-{\sqrt{-1}}a\phi)(x) \hspace{-0.8mm}: \hspace{-0.8mm} \Big ).
$$
As with the H\o egh-Krohn model,
this quantum field model 
and 
generalizations of these models 
as
$\int 
\hspace{-0.8mm}
:\hspace{-0.8mm} \exp (a\phi)(x) \hspace{-0.8mm}:
\hspace{-0.8mm}
\nu(da)$
and 
$\int 
\hspace{-0.8mm}
:\hspace{-0.8mm} \cos (a\phi)(x) \hspace{-0.8mm}:
\hspace{-0.8mm}
\nu(da)$
have also been 
studied for a long period
by many authors,
where
$\nu(da)$ is a positive Radon measure on the interval $(-{\sqrt{4\pi}},
{\sqrt{4\pi}})$ and not necessarily a Dirac measure.
See e.g., \cite{AHk73, AHk, F, FS76, FP, AHPRS,AHPRS89b}. 
All this 
gives us much motivation to study SQEs and the 
associated Dirichlet forms for such non-polynomial models. 
In the case where $\mathbb R^{2}$ and $\mathbb T^{2}$ are replaced by
$\mathbb R$, this has been discussed in \cite{AKR12}, which provides solutions of the SQE
and proves strong uniqueness of the corresponding Dirichlet operator
for any $\gamma \geq 0$. For the same
problem where ${\rm exp}(\phi)_{1}$ is replaced by $P(\phi)_{1}$, see \cite{Iwa, KR07}.

The main purpose of the present paper is to study regularized
SQEs both for 
the H\o egh-Krohn model 
\begin{equation}
\partial_{t}X_{t}(x)=-\frac{1}{2}(1-\Delta)^{1-\gamma} X_{t}(x) 
-\frac{a}{2}(1-\Delta)^{-\gamma}
:\hspace{-0.5mm}{\rm{exp}}(aX_{t}(x)) \hspace{-0.5mm}: 
+(1-\Delta)^{-\frac{\gamma}{2}}
\xi_{t}(x),
\quad t>0, x\in \mathbb T^{2},
\label{Miha-SQE}
\end{equation}
and for the Sine-Gordon model
\begin{equation}
\partial_{t}X_{t}(x)=-\frac{1}{2}(1-\Delta)^{1-\gamma} X_{t}(x) 
+\frac{a}{2}(1-\Delta)^{-\gamma}
:\hspace{-0.5mm}{\rm{sin}}(aX_{t}(x)) \hspace{-0.5mm}: 
+(1-\Delta)^{-\frac{\gamma}{2}}
\xi_{t}(x),
\quad t>0, x\in \mathbb T^{2}
\label{SG-SQE}
\end{equation}
by the Dirichlet form approach. Furthermore, we prove strong uniqueness
of the corresponding Dirichlet operators
in the relevant $L^{p}$-spaces for all $p\geq 1$.
To the best of our knowledge, 
except $L^{1}$-uniquness obtained by \cite{Wu}, 
there seems to be only a few results on $L^{p}$-uniqueness of
Dirichlet operators associated with (\ref{Miha-SQE}) and (\ref{SG-SQE})
for general $p\geq 1$.
%
We mention here that  Mihalache \cite{M} first 
proved the unique existence of the regularized SQE (\ref{Miha-SQE}) in the 
weak probabilistic sense
under restrictive conditions 
on $0<\gamma \leq 1$ and the 
charge parameter $a$,
by combining several properties of the ${\rm{exp}}(\phi)_{2}$-measure 
$\mu_{\sf exp}^{(a)}$ established in
\cite{AHk} with approximation methods in \cite{DT-Note}.
Quite recently, influenced by the quick development on singular SPDEs (cf. \cite{H, GIP}), 
unique probabilistic strong solutions to the original SQEs (\ref{Miha-SQE}) and (\ref{SG-SQE}) (i.e., with $\gamma=0$) 
were constructed by \cite{Garban, HKK, ORW19} and \cite{HS, CHS19}, respectively.
However, these results do not imply strong uniqueness of the corresponding Dirichlet operator in general
(see Remark \ref{compare} below).
Let us also mention that a recent approach on stochastic quantization of the ${\rm{exp}}(\phi)_{2}$-model
on $\mathbb T^{2}$ or $\mathbb R^{2}$, 
and of the $P(\phi)_{2}$-model on $\mathbb T^{2}$,
through elliptic SPDEs, has been developed in \cite{ADG19}
(based on previous work on dimensional reduction \cite{ADG18}).

The structure of the present paper is as follows: In Section 2, we describe our framework and state
the main results. To summarize them, let us stress that both the 
${\rm{exp}}(\phi)_{2}$-measure
and the ${\rm{cos}}(\phi)_{2}$-measure
have been 
constructed for values of the charge constant $a$ in
the interval $(-{\sqrt{4\pi}},
{\sqrt{4\pi}})$. Our methods permit to cover the whole range of this parameter and show that the pre-Dirichlet
form associated with (\ref{Miha-SQE})
(resp. (\ref{SG-SQE})), i.e., associated with the ${\rm{exp}}(\phi)_{2}$-measure 
$\mu^{(a)}_{\sf exp}$ (resp. the ${\rm{cos}}(\phi)_{2}$-measure $\mu^{(a)}_{\sf cos}$), for all 
$a \in (-{\sqrt{4\pi}},{\sqrt{4\pi}})$
is closable, and one has $L^{p}$-uniqueness of the corresponding 
Dirichlet operator for all $1\leq p < \frac{1}{2}
(1+\frac{4\pi \gamma}{a^{2}})$ with $\vert a \vert <{\sqrt{4\pi \gamma}}$.
As a corollary we show the existence of 
diffusion processes solving (\ref{Miha-SQE}) and (\ref{SG-SQE}) weakly for quasi every
starting point in the state space $H^{-\delta}(\mathbb T^{2})$ for suitable $\delta>0$.
In Section 3, we construct the Wick power, the Wick exponential and the Wick trigonometric mappings, 
including also a new detailed 
estimate for Wick powers (Proposition \ref{Def-Wick}) and
a proof of the finite dimensional approximation of the Wick exponential and the Wick trigonometric mappings
(Theorem \ref{Wick-Exp} and Corollary \ref{Wick-Tri}). 
In Section 4, we provide the proof of our main result (Theorem \ref{ES}) by combining the tools of Section 3 with
the argument in \cite{LR}. Finally, in the appendix, we present 
key estimates on Green's function, which are used in Sections 3 and 4.

Throughout the present paper, we denote by
$C$ 
positive constants which
may change from line to line.
When their dependence on some parameters are significant, we specify them
as
$C_{p}$, $C(\varepsilon)$, etc.
Besides, we use 
the notation $ a \lesssim b$ if there exists a positive constant $C$, independent of the variables under consideration,
such that $a \leq C\cdot b$.
\section{Framework and results}

We begin by introducing some notations and objects we will be working with.
Let ${\mathbb T}^{2}=(\mathbb R/ 2\pi \mathbb Z)^{2}$ be the 2-dimensional torus equipped with
the Lebesgue measure $dx$. 
Let $L^{2}({\mathbb T}^{2}; {\mathbb C})$ 
be the Hilbert space consisting all $\mathbb C$-valued
Lebesgue square integrable functions equipped with the usual inner product
$$ \big (f,g \big)_{L^{2}}=\int_{\mathbb T^{2}} f(x){\overline{g(x)}}dx, \quad f,g\in L^{2}({\mathbb T}^{2}; {\mathbb C}).$$
For ${k}=(k_{1}, k_{2}) \in {\mathbb Z}^{2}$ and $x=(x_{1}, x_{2}) \in {\mathbb T}^{2}$,
we write $\vert k \vert=(k_{1}^{2}+k_{2}^{2})^{1/2}$, $\Vert k \Vert=\max \{ \vert k_{1} \vert, \vert k_{2} \vert \}$ 
and
$k\cdot x=k_{1}x_{1}+k_{2}x_{2}$.
We define the space of distributions $\mathcal{S}' 
(\mathbb T^{2}; {\mathbb C})$ by the set of linear maps $u$ from $\mathcal{S}
(\mathbb T^{2}; {\mathbb C})=C^{\infty}({\mathbb T}^{2}; \mathbb C)$ to $\mathbb C$, such that there 
exist $N\in \mathbb N$
and $C>0$ with
$$ \vert u(\varphi) \vert \leq C \max_{0\leq  i+j  \leq N}  \max_{x\in {\mathbb T}^{2}}
\Big \vert \frac{\partial^{i+j} \varphi}{\partial x_{1}^{i}  \partial x_{2}^{j} }
(x) \Big \vert, \quad 
\varphi \in \mathcal{S}(\mathbb T^{2}; {\mathbb C}).
 $$
Since $\mathcal{S}(\mathbb T^{2}; {\mathbb C}) \subset L^{2}(\mathbb T^{2}; {\mathbb C}) \subset 
{\mathcal S}'(\mathbb T^{2};\mathbb C)$,
the $L^{2}$-inner product $\big( \cdot, \cdot \big)_{L^{2}}$ is naturally 
extended to the pairing of 
$\mathcal{S}(\mathbb T^{2}; {\mathbb C})$
and its dual space 
${\mathcal S}'(\mathbb T^{2};\mathbb C)$.
Let $\{ {\bf e}_{k}; k \in \mathbb Z^{2} \}$ be the usual complete 
orthonormal system (CONS) of $L^{2}({\mathbb T}^{2}; \mathbb C)$
consisting of the $\mathbb C$-valued functions
$$
{\bf e}_{k}(x)=\frac{1}{2\pi}e^{{\sqrt{-1}}k\cdot x}, \quad  k\in {\mathbb Z^{2}},~x\in \mathbb T^{2}. 
$$
For each $u \in \mathcal{S}' 
(\mathbb T^{2}; \mathbb C)$, we define the Fourier transform $\hat{u}: {\mathbb Z}^{2} \to {\mathbb C}$ by
$$\hat{u}(k):=
\subscripts
 {\mathcal{S}'}
{ \langle  u, {\bf e}_{k} \rangle}
{\mathcal{S}},
\quad k\in {\mathbb Z}^{2}.
$$

We define the real $L^{2}$-Sobolev space of order $s$ with periodic boundary condition
by
$$ 
H^{s}(\mathbb T^{2})=\left\{ u\in {\mathcal S}'(\mathbb T^{2}; {\mathbb C}) \hspace{0.5mm} ; 
\sum_{k \in {\mathbb Z}^{2}} (1+\vert k \vert^{2})^{s} 
\vert {\hat u} (k) \vert^{2}
<\infty,~~{\hat u}(-k)={\overline{{\hat u}(k)}},~k\in {\mathbb Z}^{2} 
\right\}, \quad s\in \mathbb R.
$$
This space is a Hilbert space equipped with the inner product
$$ { (}u,v{ )}_{H^{s}}=
\sum_{k \in {\mathbb Z}^{2}} (1+\vert k \vert^{2})^{s}   
{\hat u}(k)
{\hat v}(-k),
\quad u,v \in H^{s}(\mathbb T^{2}).$$
Note that $H^{0}({\mathbb T}^{2})$ coincides with $L^{2}(\mathbb T^{2}):=L^{2}({\mathbb T}^{2}; {\mathbb R})$ and
we regard $H^{-s}({\mathbb T}^{2})$ as the dual space of $H^{s}({\mathbb T}^{2})$
through the standard chain
$$ H^{s}({\mathbb T}^{2})  \subset 
L^{2}({\mathbb T}^{2}) \subset 
H^{-s}({\mathbb T}^{2}), \qquad s \geq 0.$$
We denote by $\subscripts
 {H^{-\alpha}}
{\langle \cdot,  \cdot  \rangle}
{H^{\alpha}}
$ 
the dualization between $H^{\alpha}(\mathbb T^{2})$ and its dual space 
$H^{-\alpha}(\mathbb T^{2})$. 

Let $\Delta$ be the Laplacian with periodic boundary condition
acting on $L^{2}(\mathbb T^{2})$. It is a self-adjoint negative
operator in $L^{2}(\mathbb T^{2})$ with
${\rm Dom}(\Delta)=H^{2}({\mathbb T}^{2})$. We set $A:=(1-\Delta)^{-1}$
and $\lambda_{k}:=1+\vert k \vert^{2}$ ($k\in {\mathbb Z}^{2}$).
Note that $A$ is a compact 
and self-adjoint operator acting on $L^{2}(\mathbb T^{2})$
and that the spectra of both $1-\Delta$ and $A$ consist only of eigenvalues.
Let $\{ e_{k}; k \in \mathbb Z^{2} \}$ be the usual CONS of $L^{2}({\mathbb T}^{2})$
associated with $\{ {\bf e}_{k}; k \in \mathbb Z^{2} \}$, that is,
$e_{(0,0)}(x)=(2\pi)^{-1}$ and
\begin{equation*}
e_{k}(x)
=
\frac{1}{{\sqrt 2}\pi}
\begin{cases}
\displaystyle{
\cos(k\cdot x)
},
& \text{ $k\in {\mathbb Z}^{2}_{+}$}
\\ 
\displaystyle{
\sin(k\cdot x)
},
& \text{ $k\in {\mathbb Z}^{2}_{-}$},
\end{cases}
\quad x\in {\mathbb T}^{2},
\label{CONS}
\end{equation*}
where ${\mathbb Z}^{2}_{+}=\{ (k_{1}, k_{2}) \in {\mathbb Z}^{2} ; \hspace{0.5mm} k_{1}>0 \} 
\cup \{ (0, k_{2}) ; \hspace{0.5mm} k_{2}>0 \}$ and ${\mathbb Z}^{2}_{-}=-{\mathbb Z}^{2}_{+}$.
Then we have that $\{(\lambda_{k}, e_{k}); k \in \mathbb Z^{2} \}$ 
and $\{(\lambda_{k}^{-1}, e_{k}); k \in \mathbb Z^{2} \}$ 
are normalized eigenbasis for  
$1-\Delta$ and $A$, respectively. Namely, we have 
$$ (1-\Delta) e_{k}=\lambda_{k}e_{k},~~Ae_{k}=\lambda_{k}^{-1} e_{k}, \qquad 
k \in \mathbb Z^{2}.$$
Setting $e_{k}^{(s)}:=\lambda^{-s/2}_{k} e_{k}$ ($s \in \mathbb R$), 
we obtain a CONS $\{ e_{k}^{(s)}; k\in \mathbb Z^{2} \}$ of $H^{s}(\mathbb T^{2})$.
Then the operator $A$ can be extended to $H^{-s}(\mathbb T^{2})$ $(s>0)$ by setting
$Ae_{k}^{(-s)}:= \lambda_{k}^{-1} e_{k}^{(-s)}$ ($k\in \mathbb Z^{2}$).
Note that $A$ is not a trace class operator because of
$$ {\rm Tr}(A)=\sum_{k \in \mathbb Z^{2}} \lambda_{k}^{-1}=
\sum_{k \in \mathbb Z^{2}} \frac{1}{1+\vert k \vert^{2}}
=\infty.$$
On the other hand, 
it follows from 
\begin{equation*}
{\rm Tr}(A^{p})=\sum_{k \in \mathbb Z^{2}} \lambda_{k}^{-p}
=
\sum_{k \in \mathbb Z^{2}} \frac{1}{(1+\vert k \vert^{2})^{p}}
<\infty,  \quad p>1
\end{equation*}
that 
$A^{p}$ is a trace class operator for any $p>1$.

Throughout the present paper, we fix two parameters $\delta(>0)$ and $\gamma$ such that
\begin{equation}
0<\gamma \leq 1,~
\delta+2\gamma >2.
\label{index-assume}
\end{equation}
We 
set $E:=H^{-\delta}({\mathbb T}^{2})$, $H:=L^{2}({\mathbb T}^{2})$, 
${\mathcal H}:=H^{\gamma}({\mathbb T}^{2})$
and 
define the space-time free field 
by
$\mu_{0}:=N(0,A^{1+\delta})$. Namely,
$\mu_{0}$ is the
mean zero
Gaussian probability measure supported on $E$ such that
\begin{equation}
\int_{E}  (l_{1}, z )_{E}(l_{2}, z )_{E} \hspace{1mm}
\mu_{0}(dz)={\big (} A^{1+\delta}l_{1}, l_{2}{\big )}_{E}, \quad l_{1},l_{2}\in E.
\label{Gauss}
\end{equation}
We can read (\ref{Gauss}) as
\begin{equation}
\int_{E} \langle z, l_{1} \rangle \langle z, l_{2} \rangle 
\hspace{1mm}
\mu_{0}(dz) 
={\big (} (1-\Delta)^{-1}l_{1}, l_{2}{\big )}_{H},
\quad l_{1},l_{2}\in E^{*} \subset H,
\label{Gauss2}
\end{equation}
where $\langle \cdot, \cdot \rangle$ denotes the dualization between 
$E$ and $E^{*}=H^{\delta}({\mathbb T}^{2})$.
%

Next, we introduce the {\it{${\rm exp}(\phi)_{2}$-measure}} 
$\mu_{\sf exp}^{(a)}$ and the {\it{${\rm cos}(\phi)_{2}$-measure}}
$\mu_{\sf cos}^{(a)}$ on $E$.
The details of
the rigorous construction of these probability measures 
as well as various properties of them will be 
presented in Section 3. 
For $\Lambda \subset \mathbb R^{2}$, we write $c\Lambda:=\{cx; x\in \Lambda \}$ ($c\in \mathbb R$).
Throughout the present paper, we assume
\begin{equation}
\Lambda \mbox{ is compact such that }
\Lambda=-\Lambda 
\mbox{ and $0$ is an interior point of } \Lambda.
\label{Lambda-condition}
\end{equation}
As typical examples of $\Lambda$, we may consider
$D(N)=\{ x \in {\mathbb R}^{2}; \vert x \vert_{\mathbb R^{2}} \leq N \}$ and
$Q(N)=\{ x=(x_{1}, x_{2}) \in {\mathbb R}^{2}; \vert x_{1} \vert \leq N,  
\vert x_{2} \vert \leq N \}$ ($N\in \mathbb N$).
We define the projection operator
$\Pi_{\Lambda}$ on $E$ by
\begin{equation}
(\Pi_{\Lambda}z)(x):=\sum_{k\in {\mathbb Z}^{2}} {\bf 1}_{\Lambda}(k) 
 {\hat z}(k) {\bf e}_{k}(x)
=  \sum_{k\in \Lambda \cap {\mathbb Z}^{2}} \langle z, e_{k} \rangle e_{k}(x), 
\quad z\in E, x\in {\mathbb T}^{2},
\label{lambda-projection}
\end{equation}
where ${\bf 1}_{\Lambda}$ stands for the indicator function of $\Lambda$.
For a cut-off function $g\in H$ and the charge parameter 
$a \in (-{\sqrt{4\pi}}, {\sqrt{4\pi}})$,
we define 
the {\it{Wick exponential}} 
$:\hspace{-0.5mm}{\rm{exp}}(a z)\hspace{-0.5mm}:\hspace{-0.5mm}(g)$
with respect to 
$\mu_{0}$
by
\begin{align} 
:\hspace{-0.8mm}{\rm{exp}}(a z)\hspace{-0.8mm}:\hspace{-0.5mm}(g) &=
\lim_{N \to \infty}
\int_{\mathbb T^{2}} \exp \Big( a (\Pi_{N\Lambda}z)(x)-\frac{a^{2}}{2}
{\mathbb E}^{\mu_{0}}[ 
(\Pi_{N\Lambda}z)(x)^{2}] \Big) \hspace{0.5mm} g(x)dx, 
\label{exp-intro}
\end{align}
where the right-hand side of (\ref{exp-intro})
converges in $L^{2}(\mu_{0})$ and it does not depend on the choice of $\Lambda$
(see Theorem \ref{Wick-Exp} for details). 
We note that 
the Gaussian multiplicative chaos ${\mathcal M}_{z}^{(a)}(A)$ ($A\in {\mathcal B}({\mathbb T}^{2})$)
coincides with
$:\hspace{-0.8mm}{\rm{exp}}(a z)\hspace{-0.8mm}:\hspace{-0.5mm}({\bf 1}_{A})$.
Using the Wick exponential
$:\hspace{-0.8mm}{\rm{exp}}(a z)\hspace{-0.8mm}:\hspace{-0.5mm}(g)$, we can also introduce
the {\it{Wick cosine}}
$:\hspace{-0.8mm}{\rm{cos}}(a z)\hspace{-0.8mm}:\hspace{-0.5mm}(g)$ and
the {\it{Wick sine}}
$:\hspace{-0.8mm}{\rm{sin}}(a z)\hspace{-0.8mm}:\hspace{-0.5mm}(g)$ 
by
\begin{align}
:\hspace{-0.8mm}{\rm{cos}}(a z)\hspace{-0.8mm}:\hspace{-0.5mm}(g) &=
\frac{1}{2} \Big (
:\hspace{-0.8mm} \exp ({\sqrt{-1}}a\phi)\hspace{-0.8mm}: \hspace{-0.5mm}(g)
\hspace{0.8mm}
+
:\hspace{-0.8mm} \exp (-{\sqrt{-1}}a\phi) \hspace{-0.8mm}: \hspace{-0.5mm} (g) \Big )
\nonumber \\
&=
\lim_{N \to \infty}
\int_{\mathbb T^{2}} \cos \big( a (\Pi_{N\Lambda}z)(x) \big) \exp \Big( \frac{a^{2}}{2}
{\mathbb E}^{\mu_{0}}[ 
(\Pi_{N\Lambda}z)(x)^{2}] \Big) \hspace{0.5mm} g(x)dx, 
\nonumber
\\
:\hspace{-0.8mm}{\rm{sin}}(a z)\hspace{-0.8mm}:\hspace{-0.5mm}(g) &=
\frac{1}{2} \Big (
\hspace{-0.8mm}:\hspace{-0.8mm} \exp ({\sqrt{-1}}a\phi)\hspace{-0.8mm}: \hspace{-0.5mm}(g)
\hspace{0.8mm}-
:\hspace{-0.8mm} \exp (-{\sqrt{-1}}a\phi) \hspace{-0.8mm}: \hspace{-0.5mm} (g)\Big )
\nonumber \\
&=
\lim_{N \to \infty}
\int_{\mathbb T^{2}} \sin \big( a (\Pi_{N\Lambda}z)(x) \big) \exp \Big( \frac{a^{2}}{2}
{\mathbb E}^{\mu_{0}}[ 
(\Pi_{N\Lambda}z)(x)^{2}] \Big) \hspace{0.5mm} g(x)dx,
\nonumber
\end{align}
respectively.
We then define the ${\rm exp}(\phi)_{2}$-measure $\mu_{\sf exp}^{(a)}$ 
on $E$ by
\begin{equation}
\mu_{\sf exp}^{(a)}(dz):=\frac{1}{Z_{\sf exp}^{(a)}}
\exp \Big(-
\hspace{-0.8mm}
:\hspace{-0.8mm}{\rm{exp}}(az)\hspace{-0.8mm}:
\hspace{-0.5mm}
({\bf 1}_{\mathbb T^{2}})  
\Big) \hspace{0.5mm} \mu_{0}(dz),
\label{mu-def}
\end{equation}
where 
\begin{equation}
Z_{\sf exp}^{(a)}:=\int_{E} \exp 
\Big(- 
\hspace{-0.8mm}:\hspace{-0.8mm}{\rm{exp}}(az)\hspace{-0.8mm}:\hspace{-0.5mm}( {\bf 1}_{\mathbb T^{2}} ) 
\Big) \hspace{0.5mm} \mu_{0}(dz) 
\label{Z-exp}
\end{equation}
is the normalization constant. 
%
%
%
Replacing
$:\hspace{-0.8mm}{\rm{exp}}(az)\hspace{-0.8mm}:\hspace{-0.5mm}( {\bf 1}_{\mathbb T^{2}} )$ by
$:\hspace{-0.8mm}{\rm{cos}}(az)\hspace{-0.8mm}:\hspace{-0.5mm}( {\bf 1}_{\mathbb T^{2}} )$ 
in (\ref{mu-def}) and (\ref{Z-exp}),
we may also define the $\cos (\phi)_{2}$-measure $\mu_{\sf cos}^{(a)}$.

Now we are in a position to introduce a
pre-Dirichlet form $({\mathcal E},{\mathfrak F}C_{b}^{\infty})$. We put $\mu=\mu_{\sf exp}^{(a)}, \mu_{\sf cos}^{(a)}$ 
and set $K:={\rm Span} \{e_{k}; k\in {\mathbb Z}^{2} \}$. We mention that
$K \subset E^{*}$ is a dense linear subspace of $E$.
Let ${\mathfrak F}C_{b}^{\infty}:={\mathfrak  F}C_{b}^{\infty}(K)$ be the space of all smooth cylinder functions
on $E$ having the form
$$
F(z)=f( \langle z, \varphi_{1} \rangle , \ldots , \langle z, \varphi_{n} \rangle), \quad z\in E,
$$
with $n\in {\mathbb N}$, $f\in C^{\infty}_{b}({\mathbb R}^{n}, {\mathbb R})$ 
and $\{ \varphi_{1}, \ldots , \varphi_{n} \} \subset K$.
Since we have supp$(\mu)=E$, two different functions in
${\mathfrak F}C_{b}^{\infty}$ represent two different $\mu$-classes.
Note that ${\mathfrak F}C_{b}^{\infty}$ is dense in $L^{p}(\mu)$ for all $p\geq 1$.
%
If $F: E \to \mathbb R$ is differentiable along every $h \in H$, and the mapping
$h \mapsto ({\partial F}/{\partial h})(z)$ is a bounded linear functional on $H$,
the function $F$ is called $H$-{\it{G\^ateaux differentiable}} at $z\in E$. Such a mapping
is called the $H$-{\it{G\^ateaux derivative}} of $F$ at $z\in E$, and denoted by $D_{H}F(z)$.
For a cylinder function $F\in {\mathfrak F}C^{\infty}_{b}$ having the above form,
the $H$-G\^ateaux derivative $D_{H}F: E \to H$
is given by
$$
D_{H}F(z)=\sum_{j=1}^{n} \partial_{j}f
\big (
\langle z, \varphi_{1} \rangle , \ldots , \langle z, \varphi_{n} \rangle
\big ) \varphi_{j}, \quad z\in E,
$$
where $\partial_{j}$ stands for the $j$-th partial derivative. 
For $\varphi \in K$, we also define a function $b(\mu; \varphi): E \to \mathbb R$ by
\begin{equation*}
b(\mu; \varphi)(z)
:=
\begin{cases}
\displaystyle{
a
:\hspace{-0.5mm}{\rm{exp}}(az)\hspace{-0.5mm}:(\varphi) 
},
& \text{ $\mu=\mu_{\sf exp}^{(a)}$},
\\ 
\displaystyle{
-a
:\hspace{-0.5mm}{\rm{sin}}(az)\hspace{-0.5mm}:(\varphi) 
},
& \text{ $\mu=\mu_{\sf cos}^{(a)}$}.
\end{cases}
\label{CONS}
\end{equation*}

We consider
the pre-Dirichlet form $({\mathcal E},{\mathfrak F}C_{b}^{\infty})$ 
which is given by 
$$
{\mathcal E}(F,G)=
\frac{1}{2} 
\int_{E} \big( A^{\gamma} D_{H}F(z), D_{H}G(z) \big)_{H} \hspace{0.5mm}
\mu (dz), \quad F,G\in {\mathfrak F}C_{b}^{\infty}.
$$
Then we obtain
\begin{pr} \label{IbP}
\begin{equation}
{\mathcal E}(F,G)=-\int_{E} {\mathcal L}F(z) G(z) \mu(dz),
\quad F,G\in {\mathfrak F}C_{b}^{\infty},
\label{IbP3}
\nonumber
\end{equation}
where ${\mathcal L}F \in L^{p}(\mu)$, $1\leq p<1+\frac{4\pi}{a^{2}}$
is given by
\begin{align}
{\mathcal L}F(z) =
\frac{1}{2}
& 
\sum_{i,j=1}^{n}
\partial_{i} \partial_{j} f
\big(
\langle z,
\varphi_{1} \rangle, \ldots , \langle z, \varphi_{n} \rangle
\big ) 
\cdot
(A^{\gamma} \varphi_{i}, \varphi_{j})_{H}
\nonumber \\
&
-\frac{1}{2} 
\sum_{i=1}^{n}
\partial_{i} f
\big (
\langle z,
\varphi_{1} \rangle, \ldots , \langle z, \varphi_{n} \rangle
\big )
\cdot
\big\{\langle z, A^{\gamma-1}\varphi_{j} \rangle +
b(\mu; A^{\gamma}{\varphi}_{j})(z) 
\big\}, \quad z\in E
\nonumber 
\end{align}
for $F=f( \langle \cdot, \varphi_{1} \rangle, \ldots, \langle \cdot, \varphi_{n} \rangle) \in {\mathfrak F}C^{\infty}_{b}$.
\end{pr}
\begin{re}
Since $A^{\gamma}\varphi \in K$ for any
$\gamma \in \mathbb R$ and $\varphi \in K$,
we easily see that
${\mathfrak F}C_{b}^{\infty}$ 
coincides with the space of all cylinder functions on $E$ having the form
\begin{equation}
F(z)=f( \varphi_{1}^{*}(z), \ldots , \varphi_{n}^{*}(z)), \quad z\in E,
\label{another}
\end{equation}
with $n\in {\mathbb N}$, $f\in C^{\infty}_{b}({\mathbb R}^{n}, {\mathbb R})$ 
and $\{ \varphi_{1}, \ldots , \varphi_{n} \} \subset K$.
Here ${\varphi}_{j}^{*}$ denotes the unique continuous extension of the functional
$(\varphi_{j},\cdot)_{\mathcal H}$ to $E$, that is, ${\varphi}^{*}_{j}$ is given by
$ {\varphi}^{*}_{j}(z)=
\langle A^{-\gamma}{\varphi}_{j}, z \rangle$ ($z\in E$). 
For a cylinder function $F \in {\mathfrak F}C_{b}^{\infty}$ having the above form,
the map
$D_{\mathcal H}F:E\to {\mathcal H}$ is defined by
$$
D_{\mathcal H}F(z):=\sum_{j=1}^{n} \partial_{j} f
\big (
\varphi_{1}^{*}(z), \ldots , \varphi_{n}^{*}(z)
\big )
\varphi_{j}, \quad z\in E.
$$
By the chain rule, it coincides with the usual $\mathcal H$-G\^ateaux derivative of $F$,
hence the above definition is independent of the representation of $F$ in (\ref{another}).
On the other hand, since 
the ${H}$-G\^ateaux derivative $D_{H}F$ is also given by
$$
D_{H}F(z)=\sum_{j=1}^{n} \partial_{j} f
\big (
\varphi_{1}^{*}(z), \ldots , \varphi_{n}^{*}(z)
\big )
\cdot
A^{-\gamma}\varphi_{j}, \quad z\in E,
$$
we easily see that 
the pre-Dirichlet form $({\mathcal E},{\mathfrak F}C_{b}^{\infty})$ and its associated 
pre-Dirichlet operator ${\mathcal L}$ are also represented as 
$$
{\mathcal E}(F,G)=
\frac{1}{2} 
\int_{E} \big( D_{\mathcal H}F(z), D_{\mathcal H}G(z) \big)_{\mathcal H} \hspace{0.5mm}
\mu (dz), \quad F,G\in {\mathfrak F}C_{b}^{\infty},
$$
and 
\begin{align}
{\mathcal L}F(z) 
=
\frac{1}{2} 
&
\sum_{i,j=1}^{n}
\partial_{i} \partial_{j} f
\big(
\varphi_{1}^{*}(z), \ldots , \varphi_{n}^{*}(z)
\big ) 
(\varphi_{i}, \varphi_{j})_{\mathcal H}
\nonumber \\
&
-\frac{1}{2} 
\sum_{i=1}^{n}
\partial_{i} f
\big (
\varphi_{1}^{*}(z), \ldots , \varphi_{n}^{*}(z)
\big )
\cdot
\big\{\langle z, A^{-1}\varphi_{j} \rangle +
b(\mu; {\varphi}_{j})(z) 
 \big\}, 
\quad z\in E,
\nonumber 
\end{align}
respectively.
\end{re}
\begin{re} Since $\inf_{\vert a \vert < {\sqrt{4\pi}} } \big(1+\frac{4\pi}{a^{2}} \big)=2$, we have
${\mathcal L}F\in L^{2}(\mu)$ for any 
$\vert a \vert <{\sqrt{4\pi}}$.
\end{re}

Proposition \ref{IbP} means that the operator 
(${\mathcal L}, {\mathfrak F}C_{b}^{\infty})$ is the pre-Dirichlet operator
which is associated with the pre-Dirichlet form $({\mathcal E},{\mathfrak F}C_{b}^{\infty})$.
In particular, it implies that $({\mathcal E},{\mathfrak F}C_{b}^{\infty})$ is closable in $L^{2}(\mu)$.
We denote by ${\mathcal D}({\mathcal E})$ the completion of ${\mathfrak F}C_{b}^{\infty}$
with respect to the ${\mathcal E}_{1}^{1/2}$-norm. 
(Here we use standard notations of the 
theory of Dirichlet forms, see, e.g., \cite{A, FOT, MR}.)
Then by standard theory (cf. \cite{A, AR90, FOT, MR}),
$({\mathcal E}, {\mathcal D}({\mathcal E}))$ is a Dirichlet form and 
the operator ${\mathcal L}$ has a self-adjoint extension
$({\mathcal L}_{\mu}, {\rm Dom}({\mathcal L}_{\mu}))$, called the
{\it{Friedrichs extension}}, corresponding to the 
Dirichlet form $({\mathcal E}, {\mathcal D}({\mathcal E}))$.
The semigroup $\{ e^{t{{\mathcal L}_{\mu}}} \}_{t\geq 0}$ generated by 
$({\mathcal L}_{\mu}, {\rm Dom}({\mathcal L}_{\mu}))$ in $L^{2}(\mu)$ is Markovian, i.e.,
$0\leq e^{t{{\mathcal L}_{\mu}}}F \leq 1$, $\mu$-a.e. whenever
$0\leq F \leq 1$, $\mu$-a.e. Moreover, since 
$\{ e^{t{{\mathcal L}_{\mu}}} \}_{t\geq 0}$ is symmetric on $L^{2}(\mu)$, the Markovian
property implies that
$$ \int_{E} e^{t{{\mathcal L}_{\mu}}} F(z) \mu(dz) \leq \int_{E} F(z) \mu(dz), 
\quad F\in L^{2}(\mu),~F \geq 0,~ \mu \mbox{-a.e.}
$$
Hence $\Vert e^{t{{\mathcal L}_{\mu}}}F \Vert_{L^{1}(\mu)}
\leq \Vert F \Vert_{L^{1}(\mu)}$ holds for $F\in L^{2}(\mu)$, and 
$\{e^{t{{\mathcal L}_{\mu}}} \}_{t\geq 0}$
can be extended as a family of $C_{0}$-semigroup
of contractions in $L^{p}(\mu)$ for all $p\geq 1$. 
See e.g., \cite[Theorem X.55]{rs} for details.


It is a fundamental question whether the
Friedrichs extension is the only closed extension of
$({\mathcal L}, {\mathfrak F}C^{\infty}_{b})$
generating a $C_{0}$-semigroup on $L^{p}(\mu), p\geq 1$, i.e.,
whether we have $L^{p}(\mu)$-uniqueness for $({\mathcal L}, {\mathfrak F}C^{\infty}_{b})$.
For
$p=2$, this is equivalent to the fundamental problem of essential self-adjointness of 
${\mathcal L}$ in quantum physics (cf. \cite{rs}). 
Even if $p=2$,
in general there are many 
lower semi bounded 
self-adjoint
extensions ${\widetilde {\mathcal L}}$ of ${\mathcal L}$ in 
$L^{2}(\mu)$ which therefore generate different symmetric strongly continuous
semigroups
$\{ e^{t{\widetilde {\mathcal L}}} \}_{t\geq 0}$
in $L^{2}(\mu)$.
If, however, we have $L^{p}(\mu)$-uniqueness of ${\mathcal L}$ for some
$p\geq 2$, there is hence only one semigroup which is strongly continuous 
and with generator extending ${\mathcal L}$ in $L^{p}(\mu)$.
Consequently, in this case, only one such $L^{p}$-,
hence only one such
$L^{2}$-dynamics exists, associated with 
$\mu$.

The following theorems are the main results of the present paper. 
For the notions of ``quasi-everywhere" and ``capacity", we refer to \cite{A, FOT, MR}.
\begin{tm}
\label{ES}
Let $0<\gamma \leq 1$ be the regularization constant satisfying 
{\rm{(\ref{index-assume})}} and let 
$a\in \mathbb R$ be the charge constant 
satisfying 
$\vert a \vert <{\sqrt{4\pi \gamma}}$.
Then
the pre-Dirichlet operator $({\mathcal L}, {\mathfrak F}C_{b}^{\infty})$ 
is $L^{p}(\mu)$-unique for all $1\leq p <\frac{1}{2}\big( 1+\frac{4\pi \gamma }{a^{2}} \big)$. Namely,
 there exists exactly one
$C_{0}$-semigroup in $L^{p}(\mu)$ such that its generator extends 
$({\mathcal L}, {\mathfrak F}C_{b}^{\infty})$. In particular, 
the Dirichlet form $({\mathcal E}, {\mathcal D}({\mathcal E}))$ is the unique
extension of $({\mathcal E}, {\mathfrak F}C^{\infty}_{b})$ such that 
${\mathfrak F}C^{\infty}_{b}$ is contained in the domain of
the associated generator.
\end{tm}

Moreover, by applying \cite[Theorem 6.1]{AR-SDE}, we obtain the following result as an immediate corollary
of this theorem.
\begin{co}
There exists a conservative diffusion process 
${\mathbb M}:=(\Theta, {\mathcal F}, ( {\mathcal F}_{t} )_{t\geq 0}, (X_{t} )_{t\geq 0},
\{ {\mathbb P}_{z}\}_{z\in E} )$ such that
the semigroup $\{P_{t}\}_{t\geq  0}$ 
generated by 
the unique 
extension of $({\mathcal L}, {\mathfrak F}C^{\infty}_{b})$
satisfies the following identity
for any bounded Borel measurable function $F:E \to \mathbb R$, and $t>0$:
\begin{equation}
P_{t}F(z)=\int_{\Theta} F(X_{t}(\omega)) {\mathbb P}_{z}(d\omega), \quad \mu\mbox{-}a.s.~z\in E,
\nonumber
\label{suii-hangun}
\end{equation}
in the sense that the right-hand side is a $\mu$-version for the $L^{2}(\mu)$-class $P_{t}F$.
Moreover, 
${\mathbb M}$ 
is the unique $\mu$-symmetric Hunt process with the state
space $E$ solving 
the regularized SQEs 
{\rm{(\ref{Miha-SQE})}} 
and 
{\rm {(\ref{SG-SQE})}} 
weakly 
for ${\mathcal E}$-q.e. $z\in E$,
in the cases
$\mu=\mu_{\sf exp}^{(a)}$ 
and 
$\mu=\mu_{\sf cos}^{(a)}$, 
respectively.
Namely,
there exist a set $S\subset E$ with ${\rm Cap}(S)=0$ and 
a system of independent one-dimensional
$({\mathcal F}_{t})_{t\geq 0}$-Brownian motions $\{ (B^{k}_{t})_{t\geq 0}; k\in \mathbb Z^{2} \}$ with $B_{0}^{k}=0$ 
($k\in \mathbb Z^{2}$) 
defined on 
the probability space $(\Theta, {\mathcal F}, ({\mathcal F}_{t})_{t\geq 0}, {\mathbb P}_{z})$ such that for any
$z \in E \setminus S$, 
${\mathbb P}_{z} \big( X_{t} \in E \setminus S \mbox{ for all } t\geq 0 \big)=1$ and
the diffusion process 
$X=(X_{t})_{t\geq 0}$ 
satisfies 
\begin{align}
\langle X_{t}, e_{k} \rangle =& \langle z, e_{k} \rangle 
-\frac{1}{2}
\int_{0}^{t}
\Big \{ 
(1+\vert k \vert^{2})^{1-\gamma}
\big \langle X_{s}, e_{k} \big \rangle 
+
(1+\vert k \vert^{2})^{-\gamma}
b(\mu; e_{k})(X_{s})
\Big \} ds
\nonumber 
\\
&+
(1+\vert k \vert^{2})^{-\gamma/2} B_{t}^{k}
\qquad t>0,~k\in {\mathbb Z}^{2}, 
\nonumber
\end{align}
${\mathbb P}_{z}$-almost surely.
\end{co}
\begin{re}
\label{compare}
Following arguments in {\rm{\cite{Garban, HKK, ORW19, HS, CHS19}}}, we might construct 
a unique strong solution
for SQE {\rm{(\ref{Miha-SQE})}} (resp.
SQE {\rm{(\ref{SG-SQE})}}) for $\mu^{(a)}_{\sf exp}$- (resp. $\mu^{(a)}_{\sf cos}$-) almost every initial datum. 
However, this does not imply the $L^{p}$-uniqueness of the corresponding Dirichlet
operator. 
This is obvious, since a priori the latter might have extensions which generate 
%
%
semigroups which have no probabilistic interpretation as transition probabilities of Markovian processes.
For example, there might be an extension generating a non-Markovian semigroup which
cannot be the transition probability semigroup of a Markov process or there might be
also an extension generating a Markovian semigroup whose associated Dirichlet form is
not quasi-regular, hence by the main result in \cite{AMR93} and \cite{MR} 
this semigroup can also not be the transition probability
semigroup of a Markov process with c\`adl\`ag paths.
Therefore, in general, none of the properties $L^{p}$-uniqueness of the Dirichlet operator
and strong uniqueness of the corresponding SQE
implies the other. We refer to 
{\rm{\cite[Sections 2 and 3]{Cornell}}}
for a detailed discussion of this point.
\end{re}
\section{Preliminaries} 
\subsection{A review of the Wick power}
We give notations and review some statements
on the Wick power based on \cite{DT-Note}. 
In particular, we discuss the convergence
of the finite dimensional approximation of the Wick power 
in a strengthened form.

Let $\Lambda$ be a subset of $\mathbb R^{2}$ satisfying (\ref{Lambda-condition}) and $\Pi_{\Lambda}$
be the projection operator defined in (\ref{lambda-projection}).
We set 
$$ \rho_{\Lambda}:=
{\mathbb E}^{\mu_{0}} \big [ (\Pi_{\Lambda} z)(x)^{2} \big]^{1/2}
=
\frac{1}{2\pi} \Big \{
\sum_{k\in \mathbb Z^{2}} {\bf 1}_{\Lambda}(k)(1+\vert k \vert^{2})^{-1} \Big \}^{1/2},
\quad 
x\in {\mathbb T}^{2},$$
and define $\eta_{\Lambda,x}\in E^{*}$ by
\begin{equation}
\eta_{\Lambda,x}:=\frac{1}{\rho_{\Lambda}}\sum_{k\in \mathbb Z^{2}} {\bf 1}_{\Lambda}(k) e_{k}(x)e_{k}
=\frac{1}{\rho_{\Lambda}}\sum_{k\in \mathbb Z^{2}} 
{\bf 1}_{\Lambda}(k) 
{\overline{{\bf e}_{k}(x)}}{\bf e}_{k}, \quad x\in {\mathbb T}^{2}.
\label{eta}
\end{equation}
Then we easily have
$$(\Pi_{\Lambda}z)(x)= \rho_{\Lambda} 
{\big \langle} z, \eta_{\Lambda,x} {\big \rangle}, \qquad z\in E,~ x\in {\mathbb T}^{2}.
$$

We next consider the following Laplace equation on $\mathbb T^{2}$:
\begin{equation}
(1-\Delta)^{\alpha} u=f, \qquad \alpha>0.
\label{green-eq}
\end{equation}
If $f=\sum_{k\in {\mathbb Z}^{2}} f_{k}e_{k} \in H^{s}(\mathbb T^{2})$, then 
$$u=(1-\Delta)^{-\alpha}f=\sum_{k\in {\mathbb Z}^{2}} (1+\vert k \vert^{2})^{-\alpha} 
f_{k} e_{k}
\in H^{s+2\alpha}(\mathbb T^{2})$$ solves (\ref{green-eq}), and it
can be also represented as
$$ (1-\Delta)^{-\alpha}f(x)=\int_{\mathbb T^{2}} G^{(\alpha)}_{\mathbb T^{2}}(x,y) f(y)dy, \qquad x\in {\mathbb T}^{2}$$ 
by using the Green function $G^{(\alpha)}_{\mathbb T^{2}}(x,y)$.
We mention that
\begin{equation}
G^{(\alpha)}_{\mathbb T^{2}}(x,y)
=\sum_{k\in \mathbb Z^{2}} G^{(\alpha)}(x, y+2\pi k), \qquad x,y \in \mathbb T^{2}.
\label{kasaneawase}
\end{equation}
See e.g., \cite[Proposition 7.3.1]{GJ} for details.
Combining Proposition \ref{Green} with (\ref{kasaneawase}), we have
\begin{equation}
K^{(\alpha)}(x-y):=G^{(\alpha)}_{\mathbb T^{2}}(x,y)
\leq
\begin{cases}
\displaystyle{C ( 1+\vert x-y \vert_{\mathbb R^{2}}^{2\alpha-2})}
& \text{ if $0<\alpha<1$}
\vspace{1.5mm} \\ 
\displaystyle{
-\frac{1}{2\pi} \log \vert x-y \vert_{\mathbb R^{2}} +C
}
& \text{ if $\alpha=1$},
\end{cases}
\label{period-green-estimate}
\end{equation}
for some constant $C>0$.
%
This estimate implies $K^{(1)}\in L^{\infty -}({\mathbb T}^{2}):=\cap_{p>1} L^{p}(\mathbb T^{2})$ and
$K^{(\alpha)}\in L^{1+}({\mathbb T}^{2}):=\cup_{p>1} L^{p}(\mathbb T^{2})$ for each $0<\alpha<1$.
We set
\begin{equation}
K^{(\alpha)}_{\Lambda}(x):=\frac{1}{2\pi} \sum_{k\in \mathbb Z^{2}} {\bf 1}_{\Lambda}(k) (1+\vert k \vert^{2})^{-\alpha} 
{\bf e}_{k}(x),
\quad x\in {\mathbb T}^{2}, ~0<\alpha \leq 1.
\label{KN-delta}
\end{equation}
By \cite[Theorems 4.1 and 4.4]{Weisz}, 
for every 
$0<\alpha \leq 1$, 
the $N$-th cubic partial sum $K^{(\alpha)}_{N}\:=K^{(\alpha)}_{Q(N)}$ satisfies
\begin{equation}K^{(\alpha)}(x)=\lim_{N\to \infty} K^{(\alpha)}_{N}(x)~~~
\mbox{ for 
almost everywhere} ~x\in {\mathbb T}^{2},
\label{Weisz-2}
\end{equation}
and
\begin{equation}
\sup_{N\in \mathbb N} \big \Vert K^{(1)}_{N} \big \Vert_{L^{p}} \leq C_{p} \Vert
K^{(1)} \Vert_{L^{p}}, \quad p\geq 1,
\label{Weisz-1}
\end{equation}
for some constant $C_{p}>0$, and 
\begin{equation} 
\sup_{N\in \mathbb N} \big \Vert K^{(\alpha)}_{N} \big \Vert_{L^{1+\varepsilon}} \leq C_{\varepsilon} \Vert
K^{(\alpha)} \Vert_{L^{1+\varepsilon}}, \quad 0<\alpha<1,
\label{Weisz-11}
\end{equation}
where $\varepsilon>0$
is some sufficiently small constant. 
%
%

Next, we introduce a mapping $W: E^{*}\ni f \mapsto W_{f}\in L^{2}(\mu_{0})$ by
$$ W_{f}(z):=\langle z,f \rangle, \quad z\in E.$$
It follows from (\ref{Gauss2}) that $W_{f}\in L^{2}(\mu_{0})$ and 
\begin{equation} \Vert W_{f} \Vert_{L^{2}(\mu_{0})}=\Vert f \Vert_{H^{-1}},
\quad f\in E^{*}.
\nonumber
\label{isometry}
\end{equation}
Therefore the mapping $W_{f}$ gives rise to an isometric isomorphism from
$H^{-1}(\mathbb T^{2})$ into $L^{2}(\mu_{0})$. 
We still denote the mapping from 
$H^{-1}(\mathbb T^{2})$ into $L^{2}(\mu_{0})$
by $W_{f}$ and call it the {\it{white noise function}}.

Let $\{ {\mathbb H}_{n} \}_{n=0}^{\infty}$ be the Hermite polynomials defined by
\begin{eqnarray}
{\mathbb H}_{n}(\xi)&:=& \frac{(-1)^{n}}{{\sqrt{n!}}}e^{\xi^{2}/2} \big( \frac{d}{d\xi} \big)^{n}e^{-\xi^{2}/2},
\quad n=0,1,2,\ldots,~ \xi \in \mathbb R.
\nonumber
\end{eqnarray}
For later use, we recall the well-known formula
\begin{eqnarray}
& &
\int_{E} {\mathbb H}_{m}(W_{f}(z)) {\mathbb H}_{n} (W_{g}(z)) \mu_{0}(dz)
=\delta_{m,n} (f,g)_{H^{-1}}^{n},
\label{Wick-2}
\end{eqnarray}
where $f,g\in H^{-1}(\mathbb T^{2})$ with $\Vert f \Vert_{H^{-1}}
=\Vert g \Vert_{H^{-1}}=1$. 
It follows from (\ref{Wick-2}) that ${\mathbb H}_{n}(W_{f})\in L^{2}(\mu_{0})$ 
for any $f\in H^{-1}(\mathbb T^{2})$ and $n\in \mathbb N \cup \{0 \}$.

For a real separable Hilbert space ${\mathcal K}$,
we denote by 
${\mathcal C}_{n}({\mathcal K})$, $n\in \mathbb N$, 
the closed linear subspace
of $L^{2}(\mu_{0}; {\mathcal K})$ generated by the $\mathcal K$-valued random variables 
$\{ {\mathbb H}_{n}(W_{f})k ;~ f\in H^{-1}(\mathbb T^{2}), \Vert f \Vert_{H^{-1}}=1, 
k\in {\mathcal K} \}$
and by ${\mathcal C}_{0}(\mathcal K)$ the set of constant $\mathcal K$-valued vectors. 
We set ${\mathcal C}_{n}:={\mathcal C}_{n}({\mathbb R})$ for simplicity.
The space ${\mathcal C}_{n}(\mathcal K)$
is called the {\it{Wiener chaos}} of order $n$
and the following
It\^o-Wiener decomposition holds:
\begin{equation}
L^{2}(\mu_{0}; {\mathcal K})={\oplus}_{n=0}^{\infty} {\mathcal C}_{n}(\mathcal K).
\label{IW-dec}
\nonumber
\end{equation}
Applying the Nelson hypercontractivity estimate 
(cf. \cite[Theorem I.22]{Si}), we also have
\begin{equation}
\Vert F \Vert_{L^{p}(\mu_{0}; {\mathcal K})}
\leq ( p-1 )^{n/2} 
\Vert F \Vert_{L^{2}(\mu_{0}; {\mathcal K})}, \qquad F\in {\mathcal C}_{n}({\mathcal K}),~ p\geq 2.
 \label{Nelson-Hyper}
\end{equation}

We now introduce the renormalized powers. 
For any $n\in {\mathbb N} \cup \{0 \}, N\in \mathbb N$ and
$x\in {\mathbb T}^{2}$, we define the {\it{$\rho_{\Lambda}^{n}$-regularized Wick power evaluated 
at $x\in {\mathbb T}^{2}$}} by
\begin{align}
{\Phi}_{n,\Lambda}(z)_{x}&:={\sqrt{n!}} \rho_{\Lambda}^{n}{\mathbb H}_{n} {\Big (} 
\rho_{\Lambda}^{-1}
(\Pi_{\Lambda} z)(x) {\Big ) }
\nonumber \\
&=
{\sqrt{n!}} \rho_{\Lambda}^{n}{\mathbb H}_{n} {\big ( } \langle z, \eta_{\Lambda,x} \rangle {\big ) }
={\sqrt{n!}} \rho_{\Lambda}^{n}{\mathbb H}_{n} {\big ( } W_{\eta_{\Lambda, x}}(z) {\big ) }, \qquad z\in E.
\label{phi-def}
\end{align}
By (\ref{eta}) and
(\ref{KN-delta}), we have
\begin{equation} \big ( \eta_{\Lambda,x}, \eta_{\Lambda',y} \big)_{H^{-1}}=\rho_{\Lambda}^{-1}\rho_{\Lambda'}^{-1} 
K^{(1)}_{\Lambda}(x-y), \quad \Vert \eta_{\Lambda,x} \Vert_{H^{-1}}=1, \qquad x,y\in {\mathbb T}^{2}
\label{key-eta}
\end{equation}
for any two subsets $\Lambda \subset \Lambda'$ in $\mathbb R^{2}$ satisfying
(\ref{Lambda-condition}).
Thus it follows from (\ref{Wick-2}), (\ref{phi-def}) and (\ref{key-eta})
that
\begin{equation}
\Vert \Phi_{n,\Lambda}(\cdot)_{x} \Vert_{L^{2}(\mu_{0})}={\sqrt{n!}}\rho_{\Lambda}^{n}, \qquad  
n\in {\mathbb N} \cup \{0 \}, x\in \mathbb T^{2}.
\label{phi-g-kasekibun}
\end{equation}
We now introduce
\begin{equation}
(\Phi_{n,\Lambda}(\cdot), g)_{H}:=
\int_{\mathbb T^{2}} \Phi_{n,\Lambda}(\cdot)_{x}g(x) dx \quad \mbox{ in }L^{2}(\mu_{0}), \quad g\in H.
\label{phi-g-naiseki}
\end{equation}
Note that the right-hand side of (\ref{phi-g-naiseki}) is well-defined in the sense of Bochner integrals
because of (\ref{phi-g-kasekibun}).
Since ${\mathcal C}_{n}$ 
is a closed subspace of $L^{2}(\mu_{0})$ and 
$\Phi_{n,\Lambda}(\cdot)_{x}\in {\mathcal C}_{n}$ for
every $x\in \mathbb T^{2}$, we also see that
$(\Phi_{n,\Lambda}(\cdot), g)_{H}\in {\mathcal C}_{n}$ and 
thus
$\Phi_{n,\Lambda}\in {\mathcal C}_{n}(H^{-\alpha}(\mathbb T^{2}))$ for all $\alpha \geq 0$.

Although the following proposition might be well-known (cf. \cite[Section V.6]{Si}), we reproduce
the proof 
for the sake of completeness and later purpose.
\begin{pr}
\label{Def-Wick}
{\rm (1)}~Let $p\geq 1$ and $\alpha>0$. 
Then $\{{\Phi}_{n,N\Lambda}; N\in \mathbb N \}$ is a Cauchy sequence in 
$L^{p}(\mu_{0};H^{-\alpha}(\mathbb T^{2}))$, 
and hence it converges to a mapping ${\Phi}_{n} \in L^{p}(\mu_{0}; H^{-\alpha}(\mathbb T^{2}))$. 
Furthermore, the limit $\Phi_{n}$ is independent of the choice of $\Lambda$.
We denote ${\Phi}_{n}(z)$, $z\in E$, by
$:z^{n}:$ and call it the Wick power of order $n$. 
Moreover, for $p\geq 2$, $n\in \mathbb N$ and $\varepsilon >0$, we have
\begin{equation}
\Vert \Phi_{n} \Vert_{L^{p}(\mu_{0}; H^{-\alpha}(\mathbb T^{2}))}
\leq 
C
(p-1)^{n/2}{\sqrt{n!}} \Big \{ 
\alpha^{-2}
\big( \frac{1+\varepsilon}{4\pi \alpha} \big)^{n-1} 
n!
+C^{n} 
 \big( \frac{1+\varepsilon}{\varepsilon} \big)^{n-1} \Big \}^{1/2},
\label{zn-est}
\end{equation}
where $C>0$ is independent of $\varepsilon >0$.
%
\\
{\rm (2)}~
For each $g\in H$, the sequence $\{({\Phi}_{n,N\Lambda}(\cdot),g)_{H}; N\in \mathbb N \}$ converges 
to a function $\Phi_{n}(\cdot)(g)$ in $L^{p}(\mu_{0})$, $p\geq 1$ as $N\to \infty$. 
Furthermore, the limit $\Phi_{n}(\cdot)(g)$ is independent of the choice of $\Lambda$.
We denote ${\Phi}_{n}(z)(g)$, $z\in E$,
by $:\hspace{-0.5mm} z^{n}\hspace{-0.5mm} :(g)$ and 
we have ${\Phi}_{n}(\cdot)(g)\in {\mathcal C}_{n}$,
and
\begin{equation}
\Vert \Phi_{n}(\cdot)(g) \Vert_{L^{p}(\mu_{0})}
\leq 
C
(p-1)^{n/2}{\sqrt{n!}} 
\Big \{ 
\big( \frac{1+\varepsilon}{4\pi} \big)^{n-1} 
n!
+C^{n} 
 \big( \frac{1+\varepsilon}{\varepsilon} \big)^{n-1} \Big \}^{1/2}
 \Vert g \Vert_{H},
\label{zng-est}
\end{equation}
where $C>0$ is independent of $\varepsilon >0$.
Furthermore, 
\begin{equation}
 {\Phi}_{n}(z)(g)=
 \subscripts
 {H^{-\alpha}}
{ \langle : \hspace{-0.5mm} z^{n}
\hspace{-0.5mm} :, g  \rangle
}
{H^{\alpha}}, \quad \mu_{0}\mbox{-a.e. }z\in E,~ g\in H^{\alpha}(\mathbb T^{2}).
\label{wick-dual}
\end{equation}
\end{pr}
\vspace{2mm}

\begin{proof}~Without loss of generality, we may assume $n\in \mathbb N$ and $p\geq 2$.
(There is nothing to prove in the case $n=0$.) 
%
By
(\ref{Weisz-2}), (\ref{Weisz-1}) and (\ref{Weisz-11}), we have
\begin{align}
\lim_{M \to \infty} \int_{\mathbb T^{2}} 
K^{(\alpha)}_{M}(x) 
K_{\Lambda}^{(1)}(x)^{n} dx &=
\int_{\mathbb T^{2}} 
K^{(\alpha)}(x) 
K_{\Lambda}^{(1)}(x)^{n} dx, 
\label{shusoku1}
\\
\lim_{N \to \infty} \int_{\mathbb T^{2}} 
K^{(\alpha)}(x) 
K_{N}^{(1)}(x)^{n} dx &=
\int_{\mathbb T^{2}} 
K^{(\alpha)}(x) 
K^{(1)}(x)^{n} dx.
\label{shusoku2}
\end{align}
It follows from an elementary inequality 
$$(a+b)^{n}\leq (1+\varepsilon)^{n-1}a^{n}+ \big(\frac{1+\varepsilon}{\varepsilon} \big)^{n-1}
b^{n}, \quad a,b, \varepsilon>0$$ 
and (\ref{period-green-estimate}) that
\begin{eqnarray}
 \int_{\mathbb T^{2}} 
K^{(\alpha)}(x)
K^{(1)}(x)^{n}dx
& \leq &C^{n}+C(1+\varepsilon)^{n-1} \int_{\vert x \vert_{\mathbb R^{2}} \leq 1} 
\vert x \vert_{\mathbb R^{2}}^{2\alpha-2} \big( -\frac{1}{2\pi} \log 
\vert x \vert_{\mathbb R^{2}} \big)^{n} dx
\nonumber \\
&\mbox{  }&+C^{n+1}
 \big( \frac{1+\varepsilon}{\varepsilon} \big)^{n-1} \int_{\vert x \vert_{\mathbb R^{2}} \leq 1} 
\vert x \vert_{\mathbb R^{2}}^{2\alpha-2} dx
\nonumber \\
&\leq& C^{n}+C
\big( \frac{1+\varepsilon}{2\pi} \big)^{n-1} (2\alpha)^{-(n+1)} n!
+C^{n+1} 
 \big( \frac{1+\varepsilon}{\varepsilon} \big)^{n-1} \cdot \big(\frac{\pi}{\alpha}\big)
 \nonumber \\
&\leq & 
 C
 \alpha^{-2}
\big( \frac{1+\varepsilon}{4\pi \alpha} \big)^{n-1} 
n!
+C^{n+1} 
\alpha^{-1}
 \big( \frac{1+\varepsilon}{\varepsilon} \big)^{n-1} , \qquad 
 \varepsilon >0.
\label{sei}
\end{eqnarray}

Now we prove item (1).  
Combining (\ref{Wick-2}), (\ref{phi-def}) with 
(\ref{key-eta}), we have 
\begin{eqnarray}
\lefteqn{
\int_{E} (\Phi_{n,M\Lambda}(z),f)_{H}(\Phi_{n,N\Lambda}(z),g)_{H} \mu_{0}(dz)}
\nonumber \\
&=& n! \rho_{M\Lambda}^{n}\rho_{N\Lambda}^{n} \int_{\mathbb T^{2}}\int_{\mathbb T^{2}} f(x)g(y) 
\big( \eta_{M \Lambda,x},
\eta_{N \Lambda,y} \big)_{H^{-1}}^{n}dxdy
\nonumber \\
&=& n! \int_{\mathbb T^{2}}\int_{\mathbb T^{2}} K_{N\Lambda}^{(1)}(x-y)^{n} f(x)g(y)dxdy, 
\qquad M\geq N,~f, g\in H. 
\label{3-11}
\end{eqnarray}
Then by
(\ref{KN-delta}), (\ref{shusoku1}) and (\ref{3-11}), we have
\begin{eqnarray}
\big \Vert \Phi_{n,N\Lambda} \big \Vert_{L^{2}(\mu_{0};H^{-\alpha}(\mathbb T^{2}))}^{2}
&=&
\int_{E} \sum_{k\in {\mathbb Z}^{2}} \Big \{
(1+\vert k \vert^{2})^{-\alpha}
{\widehat{\Phi_{n,N\Lambda}(z)_{\cdot}}}(k)
{\widehat{\Phi_{n,N\Lambda}(z)_{\cdot}}}(-k)
\Big \}
\mu_{0}(dz)
\nonumber \\
&=&n!
\sum_{k\in {\mathbb Z}^{2}}
(1+\vert k \vert^{2})^{-\alpha}
\int_{\mathbb T^{2}}\int_{\mathbb T^{2}}
K_{N\Lambda}^{(1)}(x-y)^{n}{\bf e}_{k}(x) {\bf e}_{-k}(y) dxdy
\nonumber \\
&=&
\frac{n!}{2\pi}
\sum_{k\in {\mathbb Z}^{2}}
(1+\vert k \vert^{2})^{-\alpha}
\int_{\mathbb T^{2}}\int_{\mathbb T^{2}}
K_{N\Lambda}^{(1)}(x-y)^{n}{\bf e}_{k}(x-y)dxdy
\nonumber \\
&=&
\frac{n!}{2\pi}
\lim_{M\to \infty}
\sum_{\Vert k \Vert \leq M}
(1+\vert k \vert^{2})^{-\alpha}
(2\pi)^{2}\int_{\mathbb T^{2}}
K_{N\Lambda}^{(1)}(x)^{n}{\bf e}_{k}(x)dx
\nonumber \\
&=&
4\pi^{2} n! \lim_{M \to \infty} \int_{\mathbb T^{2}} 
K^{(\alpha)}_{M}(x)
K_{N\Lambda}^{(1)}(x)^{n} dx
\nonumber \\
&=&
4\pi^{2} n!  \int_{\mathbb T^{2}} 
K^{(\alpha)}(x)
K^{(1)}_{N\Lambda}(x)^{n}  dx.
\nonumber
\end{eqnarray}

On the other hand, for any $M, N\in {\mathbb N}$, we also have
\begin{eqnarray}
\lefteqn{
\int_{\mathbb T^{2}} 
K^{(\alpha)}_{M}(x)
K_{N\Lambda}^{(1)}(x)^{n} dx
}
\nonumber \\
&=&\big(\frac{1}{2\pi} \big)^{n+1}\int_{\mathbb T^{2}}
\Big \{ \sum_{\Vert k \Vert \leq M} (1+\vert k \vert^{2})^{-\alpha} {\bf e}_{k}(x) \Big \}
\cdot
\Big \{ \sum_{l\in \mathbb Z^{2}}
{\bf 1}_{N\Lambda}(l)
(1+\vert l \vert^{2})^{-\alpha} {\bf e}_{l}(x) \Big \}^{n} dx
\nonumber \\
&=&
\big(\frac{1}{2\pi} \big)^{n+1}
\sum_{\Vert k \Vert \leq M}
(1+\vert k \vert^{2})^{-\alpha}
\sum_{l(1)\in {\mathbb Z}^{2}}
\cdots
\sum_{l(n) \in {\mathbb Z}^{2}}
\Big \{ \prod_{j=1}^{n}
{\bf 1}_{N\Lambda}(l(j))
(1+\vert l(j) \vert^{2})^{-1} \Big \}
\int_{\mathbb T^{2}} {\bf e}_{k+l(1)+\cdots+l(n)} (x)dx
\nonumber \\
&=&
\big(\frac{1}{2\pi} \big)^{n-1}
\sum_{\Vert k \Vert \leq M}
(1+\vert k \vert^{2})^{-\alpha}
\sum_{l(1)\in {\mathbb Z}^{2}}
\cdots
\sum_{l(n) \in {\mathbb Z}^{2}}
\Big \{ \prod_{j=1}^{n}
{\bf 1}_{N\Lambda}(l(j))
(1+\vert l(j) \vert^{2})^{-1} \Big \}
{\bf 1}_{k+l(1)+\cdots+l(n)=0}
\nonumber \\
&\leq &
\big(\frac{1}{2\pi} \big)^{n-1}
\sum_{k\in {\mathbb Z}^{2}}
(1+\vert k \vert^{2})^{-\alpha}
\sum_{l(1)\in {\mathbb Z}^{2}}
\cdots
\sum_{l(n)\in {\mathbb Z}^{2}}
\Big \{
\prod_{j=1}^{n}(1+\vert l(j) \vert^{2})^{-1}
\Big \}
{\bf 1}_{k+l(1)+\cdots+l(n)=0}
\nonumber \\
&=&
\lim_{N \to \infty} \Big( \lim_{M\to \infty}
\int_{\mathbb T^{2}} 
K^{(\alpha)}_{M}(x)
K_{N}^{(1)}(x)^{n} dx
\Big)
=
\int_{\mathbb T^{2}}   
K^{(\alpha)}(x) K^{(1)}(x)^{n}  dx, 
\nonumber
\end{eqnarray}
where we used (\ref{shusoku1}) and (\ref{shusoku2}) for the final line.
Thus we obtain $\Phi_{n,N\Lambda}\in {\mathcal C}_{n}(H^{-\alpha}({\mathbb T}^{2}))$ and
\begin{align}
\big \Vert \Phi_{n,N\Lambda} \big \Vert_{L^{p}(\mu_{0};H^{-\alpha}(\mathbb T^{2}))}^{2}
&\leq 
(p-1)^{n}
\big  \Vert \Phi_{n,N\Lambda} \big \Vert_{L^{2}(\mu_{0};H^{-\alpha}(\mathbb T^{2}))}^{2}
\nonumber \\
&
\leq 
4\pi^{2}n!
(p-1)^{n}
\int_{\mathbb T^{2}} 
K^{(\alpha)}(x)
K^{(1)}(x)^{n}dx, \qquad N\in {\mathbb N},~ p\geq 2,
\label{znN-est}
\end{align}
where we used (\ref{Nelson-Hyper}) for the first line.

Recalling (\ref{3-11}) again, and
repeating the same arguments as above, 
for any two integers $M\geq N$,
we also have
\begin{equation}
\big \Vert \Phi_{n,M\Lambda}-\Phi_{n,N\Lambda} \big \Vert_{L^{p}(\mu_{0};H^{-\alpha}(\mathbb T^{2}))}^{2}
\leq 4\pi^{2} n! 
(p-1)^{n}
\int_{\mathbb T^{2}} K^{(\alpha)}(x)
\Big( K^{(1)}_{M\Lambda}(x)^{n}-K^{(1)}_{N\Lambda}(x)^{n} \Big)dx.
\label{base}
\end{equation}

Now we fix a constant ${\mathfrak p}={\mathfrak p}_{\alpha}>1$ sufficiently small such that 
${\mathfrak p}(1-\alpha)<1$ and let 
${\mathfrak q}={\mathfrak q}_{\alpha}(\gg 1)$ be the
conjugate constant of ${\mathfrak p}$.
By (\ref{period-green-estimate}),
$$ \frac{1}{\mathfrak p}+\frac{1}{n{\mathfrak q}}+\frac{j}{n{\mathfrak q}}+\frac{n-1-j}{n{\mathfrak q}}=1, \qquad 
j=0, 1, \ldots, n-1,$$
the elementary binomial equality 
$ a^{n}-b^{n}=(a-b)(a^{n-1}+a^{n-2}b+\cdots +ab^{n-2}+b^{n-1})$
and the generalized H\"older inequality for $(n+1)$-factors, we have
\begin{eqnarray}
\lefteqn{ 
\hspace{-4mm}
\int_{{\mathbb T}^{2}} K^{(\alpha)}(x) \big ( K_{M\Lambda}^{(1)}(x)^{n}-K_{N\Lambda}^{(1)}(x)^{n} \big )dx
}
\nonumber \\
& &
\hspace{-7mm}
\leq  
\big \Vert K^{(\alpha)}
\big \Vert_{L^{\mathfrak p}(\mathbb T^{2})}
\sum_{j=0}^{n-1} 
\big \Vert (K_{M\Lambda}^{(1)}-K_{N\Lambda}^{(1)}) (K_{M\Lambda}^{(1)})^{j} 
(K_{N\Lambda}^{(1)})^{n-1-j}
\big \Vert_{L^{\mathfrak q}(\mathbb T^{2})}
\nonumber \\
& & 
\hspace{-7mm}
\leq 
\big \Vert K^{(\alpha)}
\big \Vert_{L^{\mathfrak p}(\mathbb T^{2})}
\sum_{j=0}^{n-1}
\big \Vert K_{M\Lambda}^{(1)}-K_{N\Lambda}^{(1)}
\big \Vert_{L^{n{\mathfrak q}}( \mathbb T^{2})}
\big \Vert K_{M\Lambda}^{(1)}
\big \Vert_{L^{n{\mathfrak q}}(\mathbb T^{2})}^{j}
\big \Vert K_{N\Lambda}^{(1)}
\big \Vert_{L^{n{\mathfrak q}}( \mathbb T^{2})}^{n-1-j}
\nonumber \\
& & 
\hspace{-7mm}
\leq 
\big \Vert K^{(\alpha)}
\big \Vert_{L^{\mathfrak p}(\mathbb T^{2})}
\cdot
n
\Big \{ \sum_{k\in \mathbb Z^{2}} \big(\frac{1}{1+\vert k \vert^{2}} \big)^{\frac{n{\mathfrak q}}{n{\mathfrak q}-1}} \Big \}
^{\frac{(n{\mathfrak q}-1)(n-1)}{n{\mathfrak q}}}
\hspace{-1mm}
\Big \{ \sum_{k\in \mathbb Z^{2}} 
{\bf 1}_{M\Lambda \setminus N\Lambda}(k)
\big(\frac{1}{1+\vert k \vert^{2}} \big)^{\frac{n{\mathfrak q}}{n{\mathfrak q}-1}} \Big \}
^{\frac{n{\mathfrak q}-1}{n{\mathfrak q}}}.
\label{K-binomial}
\end{eqnarray}
By summarizing (\ref{base}), (\ref{K-binomial}), and letting $M, N\to \infty$, we have shown that
$\{{\Phi}_{n,N\Lambda}\}_{n=1}^{\infty}$ is a Cauchy sequence in $L^{p}(\mu_{0};H^{-\alpha}(\mathbb T^{2}))$.
%
Letting $N\to \infty$ on both sides of (\ref{znN-est}) and combining it with (\ref{sei}), 
we obtain the desired estimate (\ref{zn-est}).

For a general subset $\Lambda \subset \mathbb R^{2}$ satisfying (\ref{Lambda-condition}),
we can find some $r>0$ and $R\in \mathbb N$ such that
$Q(r) \subset \Lambda \subset Q(R)$.
Repeating the same argument above, we obtain
$$ \Vert \Phi_{n, N\Lambda}-\Phi_{n, Q(NR)} \Vert^{2}_{L^{p}(\mu_{0}; H^{-\alpha}({\mathbb T}^{2}))}
\lesssim 
\big \Vert K^{(\alpha)}
\big \Vert_{L^{\mathfrak p}(\mathbb T^{2})}
\Big \{ \sum_{[Nr]<\Vert k \Vert \leq NR}
 \big(\frac{1}{1+\vert k \vert^{2}} \big)^{\frac{n{\mathfrak q}}{n{\mathfrak q}-1}} \Big \}
^{\frac{n{\mathfrak q}-1}{n{\mathfrak q}}} \to 0 \quad \mbox { as }N\to \infty.
$$
Thus we have completed the proof of item (1).
\vspace{2mm} 

Next, we prove item (2).
Applying Young's inequality for convolution, we have
\begin{align}
\big \Vert 
\Phi_{n, M\Lambda}(\cdot)(g)- 
\Phi_{n,N\Lambda}(\cdot)(g) \big \Vert^{2}_{L^{2}(\mu_{0})}
&=n! \int_{\mathbb T^{2}} \int_{\mathbb T^{2}} \big(
K_{M\Lambda}^{(1)}(x-y)^{n}-
K_{N\Lambda}^{(1)}(x-y)^{n}
\big) g(x)g(y) dxdy
\nonumber \\
&=n! \Big( \big \{
(K_{M\Lambda}^{(1)})^{n}
-
(K_{N\Lambda}^{(1)})^{n} \big \}*g, g \Big)_{L^{2}}
\nonumber \\
&\leq n! \big \Vert
(K_{M\Lambda}^{(1)})^{n}
-
 (K_{N\Lambda}^{(1)})^{n} \big \Vert_{L^{1}} \Vert g \Vert_{L^{2}}^{2}, \qquad M>N.
\nonumber
\end{align}
Then by following the proof of item (1),
we can easily show $L^{p}(\mu_{0})$-convergence of the sequence 
$\{({\Phi}_{n,N\Lambda}(\cdot),g)_{H}; N\in \mathbb N \}$. 
The fact that
$\Phi_{n}(\cdot)(g)\in {\mathcal C}_{n}$ and the estimate
(\ref{zng-est}) are clear.

For $g\in H^{\alpha}(\mathbb T^{2})$, we have
\begin{equation}
({\Phi}_{n,N}(z),g)_{H}=
\subscripts
 {H^{-\alpha}}
{\langle \Phi_{n,N}(z),g \rangle}
{H^{\alpha}}, \quad z\in E, N\in \mathbb N,
\label{nN-dual}
\end{equation}
Then by taking the limit $N\to \infty$ on both sides of (\ref{nN-dual}),
we obtain (\ref{wick-dual}). This completes the proof.
\end{proof}
\begin{pr}
\label{Taylor-expansion}
Let $p\geq 2$ and $0<\alpha \leq 1$. If 
$\vert a \vert< {\sqrt{\frac{4\pi \alpha}{p-1}}}$,
the mapping $\sum_{n=0}^{\infty} \frac{a^{n}}{n!} {\Phi}_{n}(\cdot)$ 
converges as an element of
$L^{p}(\mu_{0};H^{-\alpha}(\mathbb T^{2}))$. Furthermore, for each 
$g\in H$, the function $\sum_{n=0}^{\infty} \frac{a^{n}}{n!} ({\Phi}_{n}(\cdot),g)_{H}$
converges as an element of  
$L^{p}(\mu_{0})$ provided that $\vert a \vert< {\sqrt{\frac{4\pi}{p-1}}}$.
\end{pr}
\begin{proof} To prove the
first assertion,
it is sufficient to check
\begin{equation} \sum_{n=0}^{\infty} \frac{\vert a \vert^{n}}{n!} \Vert \Phi_{n} \Vert_{L^{p}(\mu_{0};
H^{-\alpha}(\mathbb T^{2}))} <\infty,~~{\mbox{if }}~~~ 
\vert a \vert< 
{\sqrt{ \frac{4\pi \alpha}{p-1} }}.
\label{check}
\end{equation}
Recalling (\ref{zn-est}),
we estimate the summand in (\ref{check}) by
\begin{eqnarray}
\frac{1}{n!}\Vert \Phi_{n} \Vert_{L^{p}(\mu_{0};
H^{-\alpha}(\mathbb T^{2}))}
&\lesssim &
(p-1)^{n/2} \Big( \frac{1+\varepsilon}{4\pi \alpha} \Big)^{\frac{n-1}{2}}
%
+\Big \{ \frac{C(p-1)}{n!} \Big \}^{n/2}
\Big( \frac{\varepsilon}{1+\varepsilon} \Big)^{\frac{n-1}{2}}
=:k(n)+l (n),
\nonumber
\end{eqnarray}
Then by
\begin{equation}
\lim_{n\to \infty} \big( \frac{k(n+1)}{k(n)} \big )={\sqrt{ \frac{4\pi \alpha}{(p-1)(1+\varepsilon)}} }, 
\quad 
\lim_{n\to \infty} \big( \frac{l(n+1)}{l(n)} \big )=0,
\nonumber
\end{equation}
and letting $\varepsilon \searrow 0$,
we see that (\ref{check}) holds. Replacing (\ref{zn-est}) by (\ref{zng-est}) in the above argument,
we obtain the second assertion. This completes the proof.
\end{proof}
\subsection{Finite dimensional approximation of the Wick exponential}
In this subsection, we construct the Wick exponential mapping 
$\hspace{-1mm} :\hspace{-2mm} {\rm{exp}}(a\cdot) \hspace{-2mm}:$
through the finite dimensional approximation introduced in the previous subsection.

First, we take an arbitrary number $a\in \mathbb R$ and $\Lambda \subset \mathbb R^{2}$ satisfying
(\ref{Lambda-condition}).
For any $z\in E$, we define the smooth non-negative function 
$ {\rm exp}_{\Lambda}(az)_{\bullet}: {\mathbb T}^{2} \to [0,\infty)$ by
$$ {\rm exp}_{\Lambda}(az)_{x}:=
\exp \Big \{ a (\Pi_{\Lambda}z)(x) 
-\frac{(a\rho_{\Lambda})^{2}}{2} 
\Big \}, \quad x \in {\mathbb T}^{2},
$$
where $\Pi_{\Lambda}$ is the projection operator defined in (\ref{lambda-projection}).
Substituting $t=a{\rho_{\Lambda}}$ and
$\xi={\rho_{\Lambda}}^{-1}(\Pi_{\Lambda}z)(x)$ into the formula
$$ \exp \big(-\frac{t^{2}}{2}+\xi t \big)=\sum_{n=0}^{\infty} \frac{a^{n}}{\sqrt{n!}} {\mathbb H}_{n}(\xi),
$$
we easily deduce
\begin{equation}
{\rm exp}_{\Lambda}(az)_{x}=\sum_{n=0}^{\infty} \frac{a^{n}}{n!} \Phi_{n, \Lambda}(z)_{x}, 
\quad z\in E, x\in {\mathbb T}^{2},
\label{exp-N-Taylor}
\end{equation}
with $\Phi_{n, \Lambda}(z)_{x}$ as in 
(\ref{phi-def}).
Furthermore, we have the following.
\begin{pr}
\label{pr-pr}
For any $\alpha \geq 0$, we have
${\rm exp}_{\Lambda}(a\cdot) \in L^{2}(\mu_{0}; H^{-\alpha}(\mathbb T^{2}))$ and
\begin{equation}
{\mathcal P}_{n} \Big( {\rm exp}_{\Lambda}(a\cdot) \Big)
=\frac{a^{n}}{n!}
 \Phi_{n,\Lambda}(\cdot), \quad n=0,1,2,\ldots,
\label{L2-projection}
\end{equation}
where ${\mathcal P}_{n}
:L^{2}(\mu_{0}; H^{-\alpha}(\mathbb T^{2})) \to {\mathcal C}_{n}(
H^{-\alpha}(\mathbb T^{2}))$ is the orthogonal projection onto the Wiener chaos of order $n$.
\end{pr}
\noindent
%
{\it{Proof.}}~Recalling
\begin{equation}
\int_{E} e^{
W_{f}(z)} 
\mu_{0}(dz)=\exp \big ( \frac{1}{2} \Vert f \Vert_{H^{-1}}^{2} \big), \quad f\in H^{-1}({\mathbb T}^{2}),
\label{Wick-1}
\nonumber
\end{equation}
and 
$\Vert 
\eta_{\Lambda,x} \Vert_{H^{-1}}=1$ ($x\in \mathbb T^{2}$),
we have
\begin{eqnarray}
\int_{E}  \Vert {\rm exp}_{\Lambda}(az)  \Vert^{2}_{H} \mu_{0}(dz)  
&=& e^{-(a\rho_{\Lambda})^{2}} 
\int_{\mathbb T^{2}}  \Big(
\int_{E} 
\exp \big ( 2a \rho_{\Lambda} W_{\eta_{\Lambda,x}}(z)  \big ) \mu_{0}(dz) \Big) dx
\nonumber \\
&=& e^{-(a\rho_{\Lambda})^{2}} 
\int_{\mathbb T^{2}}  
\exp \Big( 2a^{2} \rho_{\Lambda}^{2} \Vert \eta_{\Lambda,x}  \Vert_{H^{-1}}^{2} \Big ) dx
=
4\pi^{2} e^{a^{2}\rho_{\Lambda}^{2}},
\label{exp-2moment}
\nonumber
\end{eqnarray}
which means ${\rm exp}_{\Lambda}(a\cdot)\in L^{2}(\mu_{0}; H^{-\alpha}(\mathbb T^{2}))$ ($\alpha \geq 0$).
Combining this and the fact that $\Phi_{n, \Lambda} \in {\mathcal C}_{n}(H^{-\alpha}({\mathbb T}^{2}))$
with (\ref{exp-N-Taylor}), we easily get (\ref{L2-projection}). This completes the proof.
\qed

As an immediate corollary of Proposition \ref{pr-pr}, we easily obtain
\begin{equation}
\big( {\rm exp}_{\Lambda}(a\cdot), g
\big)_{H}=\sum_{n=0}^{\infty}{\frac{a^{n}}{n!}}( \Phi_{n,\Lambda}(\cdot),g )_{H}  \quad \mbox{ in }L^{2}(\mu_{0}),~g\in H.
\nonumber
\end{equation}

The following theorem plays a crucial role in the proof of our main results.
The reader is refered to \cite[Theorem V. 24 and Proposition VIII. 43]{Si} for a similar result
in the case of
space-time cut off quantum fields with interactions of exponential type.
\begin{tm} \label{Wick-Exp}
Let 
$p\geq 2$, $g\in H$ and
$\Lambda \subset \mathbb R^{2}$ be a subset of $\mathbb R^{2}$ satisfying
{\rm{(\ref{Lambda-condition})}}.
Then we have the following:
\\
{\rm (1)} 
Let $a\in (-{\sqrt{4\pi \alpha}}, {\sqrt{4\pi \alpha}})$.
Then 
\begin{equation}
\lim_{N \to \infty}
{\rm exp}_{N\Lambda}(a\cdot)
=\sum_{n=0}^{\infty} \frac{a^{n}}{n!} \Phi_{n}(\cdot) \quad in~
L^{2}(\mu_{0};H^{-\alpha}(\mathbb T^{2})).
\label{L2-Wick-conv}
\end{equation}
We call the right-hand side of 
{\rm{(\ref{L2-Wick-conv})}}
the Wick exponential mapping and 
denote it by $:\hspace{-0.8mm}{\rm{exp}}(az) \hspace{-0.8mm}:\hspace{0.5mm}$
($z\in E$). 
\\
{\rm (2)}~Let $a\in (-{\sqrt{4\pi}}, {\sqrt{4\pi}})$.
Then for each $g\in H$, 
\begin{equation}
\lim_{N \to \infty}
\big( {\rm exp}_{N\Lambda}(a\cdot), g \big)_{H}
=\sum_{n=0}^{\infty} \frac{a^{n}}{n!} \Phi_{n}(\cdot)(g) \quad in~
L^{2}(\mu_{0}).
\label{L2g-Wick-conv}
\end{equation}
We denote the 
the right-hand side of 
{\rm{(\ref{L2g-Wick-conv})}}
by 
$:\hspace{-0.8mm}{\rm{exp}}(az)\hspace{-0.8mm}: \hspace{-0.5mm}(g)$ ($z\in E$).
Furthermore, if $g\in H^{\alpha}(\mathbb T^{2})$ and 
$a\in (-{\sqrt{4\pi \alpha}}, {\sqrt{4\pi \alpha}})$, we have
\begin{equation}
:\hspace{-0.8mm}{\rm{exp}}(az)\hspace{-0.8mm}:\hspace{-0.5mm}(g)
=
\subscripts
 {H^{-\alpha}}
{\big \langle \hspace{-0.5mm}:\hspace{-0.8mm}{\rm{exp}}(az) \hspace{-0.8mm}: ,g \big \rangle}
{H^{\alpha}},
\quad \mu_{0}\mbox{-a.e. }z\in E.
\nonumber
\end{equation}
{\rm (3)}~If $\vert a \vert< {\sqrt{\frac{4\pi \alpha}{p-1}}}$,
the Wick exponential $:\hspace{-0.8mm}{\rm{exp}}(a\cdot) \hspace{-0.8mm}:$
belongs to $L^{p}(\mu_{0};H^{-\alpha}(\mathbb T^{2}))$.
Furthermore, if 
$\vert a \vert< {\sqrt{\frac{4\pi }{p-1}}}$,
the function
$:\hspace{-0.8mm}{\rm{exp}}(a\cdot) \hspace{-0.8mm}:\hspace{-0.5mm}(g)$ belongs to $L^{p}(\mu_{0})$.
\end{tm}
\vspace{2mm}

\begin{proof}
By (\ref{znN-est}) and (\ref{check})
$$ \sum_{n=0}^{\infty} \frac{\vert a \vert^{n}}{n!} \Big( \sup_{N\in {\mathbb N}} \Vert \Phi_{n,N\Lambda} 
 \Vert_{L^{2}(\mu_{0}; H^{-\alpha}(\mathbb T^{2}))} \Big)<\infty,~~\mbox{if }~~
\vert a \vert< 
{\sqrt{4\pi \alpha}}.
$$
Thus we obtain item (1) as
\begin{align}
\lim_{N \to \infty} \exp_{N\Lambda}(a\cdot) &=
\lim_{N \to \infty} \sum_{n=0}^{\infty} \frac{a^{n}}{n!} \Phi_{n, N\Lambda}(\cdot)
\nonumber \\
&=\sum_{n=0}^{\infty} \frac{a^{n}}{n!} \Big ( \lim_{N\to \infty}\Phi_{n, N\Lambda}(\cdot) \Big)
=\sum_{n=0}^{\infty} \frac{a^{n}}{n!} \Phi_{n}(\cdot) 
\quad
\mbox{in}~
L^{2}(\mu_{0};H^{-\alpha}(\mathbb T^{2})).
\nonumber
\end{align}
In the same way, we can prove the former part of item (2).
For $g\in H^{\alpha}(\mathbb T^{2})$, we have
\begin{eqnarray}
{\lefteqn{ \int_{E} 
\big \vert
\subscripts
 {H^{-\alpha}}
{\langle {\rm{exp}}_{N}(az),g \big \rangle}
{H^{\alpha}}
-
\subscripts
 {H^{-\alpha}}
{\big \langle \hspace{-0.5mm}:\hspace{-0.8mm}{\rm{exp}}(az) \hspace{-0.8mm}: ,g \big \rangle}
{H^{\alpha}}
\big \vert^{2} \mu_{0}(dz) }}
\nonumber \\
&\leq & \Vert g \Vert_{H^{\alpha}}^{2}
\int_{E} \big \Vert \hspace{0.5mm} {\rm{exp}}_{N}(az)-
:\hspace{-0.8mm}{\rm{exp}}(az) \hspace{-0.8mm}: \hspace{-0.5mm} \big \Vert_{H^{-\alpha}}^{2} \hspace{0.5mm}
\mu_{0}(dz) 
\to 0 \quad {\mbox{ as }}N\to \infty.
\nonumber
\end{eqnarray}
Hence by taking the limit $N\to \infty$ on both sides of 
\begin{equation}
({\rm{exp}}_{N}(az),g)_{H}=
\subscripts
 {H^{-\alpha}}
{\langle {\rm{exp}}_{N}(az),g \big \rangle}
{H^{\alpha}}, \quad z\in E, N\in \mathbb N,
\label{eN-dual}
\nonumber
\end{equation}
we obtain the latter part of item (2).

Since item (3) is a straightforward consequence of items (1), (2) and 
Proposition \ref{Taylor-expansion},
our proof is complete.
\end{proof}
\begin{re} In \cite[Theorem 2.2]{HKK}, it is shown that the sequence
$\{ \exp_{2^{N}\Lambda}(az) \}_{N=0}^{\infty}$ converges both in $H^{-\alpha}(\mathbb T^{2})$~
$\mu_{0}$-almost surely and in $L^{2}(\mu_{0}; H^{-\alpha}(\mathbb T^{2}))$ with
$\frac{a^{2}}{4\pi}<\alpha<1$.
\end{re}

Finally, we mention a significant property of the Wick exponential, which, however, it seems to be a well-known fact.
For the sake of completeness, we give a proof.
See also e.g., \cite[Theorem V.24]{Si} and \cite[Corollary 2.3]{HKK}.
\begin{pr} 
\label{Z-positive}
For all $\vert a \vert <{\sqrt{4\pi}}$, the function
$\exp \big(-\hspace{-0.8mm}:\hspace{-0.8mm}{\rm{exp}}(a\cdot)\hspace{-0.8mm}: 
\hspace{-0.8mm}({\bf 1}_{\mathbb T^{2}}) \big )$
belongs to $L^{\infty}(\mu_{0})$ and 
\begin{equation}
Z^{(a)}_{\sf exp}:=
\int_{E} \exp 
\Big(-\hspace{-0.8mm}:\hspace{-0.8mm}{\rm{exp}}(az)\hspace{-0.8mm}:
\hspace{-0.8mm}( {\bf 1}_{\mathbb T^{2}} ) \Big) 
\mu_{0}(dz)>0.
\nonumber
\end{equation}
\end{pr}
\begin{proof}
Since 
$\big( \exp_{\Lambda}(az), {\bf 1}_{\mathbb T}^{2} \big)_{H} \geq 0$
for every $z\in E$, we easily see
$:\hspace{-0.8mm}{\rm{exp}}(az)\hspace{-0.8mm}: \hspace{-0.5mm} ({\bf 1}_{\mathbb T^{2}})
\geq 0$ for $\mu_{0}$-a.e. $z\in E$. This implies
that the function
$\exp \big(-:\hspace{-0.8mm}{\rm{exp}}(az)\hspace{-0.8mm}: \hspace{-0.5mm} ({\bf 1}_{\mathbb T^{2}})
\big )$
belongs to $L^{\infty}(\mu_{0})$. On the other hand,
it follows from Jensen's inequality that
\begin{align}
\int_{E} \exp 
\big(-
:\hspace{-0.8mm}{\rm{exp}}(az)\hspace{-0.8mm}: \hspace{-0.5mm} ({\bf 1}_{\mathbb T^{2}})
\big )
\hspace{0.5mm}
\mu_{0}(dz) 
& \geq  \exp \Big(-\int_{E} 
:\hspace{-0.8mm}{\rm{exp}}(az)\hspace{-0.8mm}: \hspace{-0.5mm} ({\bf 1}_{\mathbb T^{2}})
\hspace{0.5mm}
\mu_{0}(dz) \Big)
\nonumber \\
& \geq  \exp \Big(- \big \Vert 
:\hspace{-0.8mm}{\rm{exp}}(az)\hspace{-0.8mm}: \hspace{-0.5mm} ({\bf 1}_{\mathbb T^{2}})
\big \Vert_{L^{\infty}(\mu_{0})} \big)>0.
\nonumber
\end{align}
This completes the proof.
\end{proof}
%
\subsection{Finite dimensional approximation of the Wick trigonometric mappings}
In this subsection, we construct the Wick trigonometric mappings 
$:\hspace{-0.8mm} {\rm{cos}}(a\cdot)\hspace{-0.8mm}:$
and 
$:\hspace{-0.8mm} {\rm{sin}}(a\cdot)\hspace{-0.8mm}:$
for later use.

For any $z\in E$ and $a\in \mathbb R$, we define the smooth functions 
$ {\rm cos}_{\Lambda}(az)_{\bullet}$ and ${\rm sin}_{\Lambda}(az)_{\bullet}$ by
\begin{align}
{\rm{cos}}_{\Lambda}(az)_{x}&:=
\frac{1}{2} \big (
{\rm{exp}}_{\Lambda}({\sqrt{-1}}az)_{x}
+
{\rm{exp}}_{\Lambda}(-{\sqrt{-1}}az)_{x}
\big )
=
 \cos \big( a (\Pi_{\Lambda}z)(x) \big) \exp \Big\{ \frac{(a\rho_{\Lambda})^{2}}{2}
\Big \},
\nonumber
\\
{\rm{sin}}_{\Lambda}(az)_{x}&:=
\frac{1}{2} \big (
{\rm{exp}}_{\Lambda}({\sqrt{-1}}az)_{x}
-
{\rm{exp}}_{\Lambda}(-{\sqrt{-1}}az)_{x}
\big )
=
 \sin \big( a (\Pi_{\Lambda}z)(x) \big) \exp \Big\{ \frac{(a\rho_{\Lambda})^{2}}{2}
\Big \}, \quad x\in \mathbb T^{2}.
\nonumber
\end{align}
Repeating the same argument as in the previous subsection, we easily obtain
the following:
\begin{co} \label{Wick-Tri}
Let 
$p\geq 2$, $g\in H$ and
$\Lambda \subset \mathbb R^{2}$ be a symmetric and compact set
as introduced in the beginning of this section.
\\
{\rm (1)} 
Let $a\in (-{\sqrt{4\pi \alpha}}, {\sqrt{4\pi \alpha}})$.
Then 
\begin{align}
\lim_{N \to \infty}
{\rm cos}_{N\Lambda}(a\cdot)
&=\sum_{n=0}^{\infty} (-1)^{n} \frac{a^{2n}}{(2n)!} \Phi_{2n}(\cdot),
\label{L2-cos-conv}
\\
\lim_{N \to \infty}
{\rm sin}_{N\Lambda}(a\cdot)
&=\sum_{n=0}^{\infty} (-1)^{n} \frac{a^{2n+1}}{(2n+1)!} \Phi_{2n+1}(\cdot)
\label{L2-sin-conv}
\end{align}
in~
$L^{2}(\mu_{0};H^{-\alpha}(\mathbb T^{2}))$.
We call the right-hand side of
{\rm{(\ref{L2-cos-conv})}} and 
{\rm{(\ref{L2-sin-conv})}} 
the Wick cosine mapping and the Wick sine mapping. We
denote them by $:\hspace{-0.5mm}{\rm{cos}}(a\cdot) \hspace{-0.5mm}:\hspace{0.5mm}$
and 
$:\hspace{-0.5mm}{\rm{sin}}(a\cdot) \hspace{-0.5mm}:\hspace{0.5mm}$, respectively.
\\
{\rm (2)}~Let $a\in (-{\sqrt{4\pi}}, {\sqrt{4\pi}})$.
Then for each $g\in H$, 
\begin{align}
\lim_{N \to \infty}
\big ({\rm cos}_{N\Lambda}(a\cdot), g \big)_{H}
&=\sum_{n=0}^{\infty} (-1)^{n} \frac{a^{2n}}{(2n)!} \Phi_{2n}(\cdot)(g),
\label{L2g-cos-conv}
\\
\lim_{N \to \infty}
\big(
{\rm sin}_{N\Lambda}(a\cdot), g \big)_{H}
&=\sum_{n=0}^{\infty} (-1)^{n} \frac{a^{2n+1}}{(2n+1)!} \Phi_{2n+1}(\cdot)(g)
\label{L2g-sin-conv}
\end{align}
in~
$L^{2}(\mu_{0})$.
We denote the 
the right-hand side of 
{\rm{(\ref{L2g-cos-conv})}} and {\rm{(\ref{L2g-sin-conv})}} 
by
$:\hspace{-0.5mm}{\rm{cos}}(a\cdot)\hspace{-0.5mm}: \hspace{-1mm}(g)$ 
and
$:\hspace{-0.5mm}{\rm{sin}}(a\cdot)\hspace{-0.5mm}: \hspace{-1mm}(g)$, respectively. 
Furthermore, if $g\in H^{\alpha}(\mathbb T^{2})$ and 
$a\in (-{\sqrt{4\pi \alpha}}, {\sqrt{4\pi \alpha}})$, we have
\begin{equation}
:\hspace{-0.5mm}{\rm{cos}}(az)\hspace{-0.5mm}:(g)
=
\subscripts
 {H^{-\alpha}}
{\langle :\hspace{-0.5mm}{\rm{cos}}(az) \hspace{-0.5mm}: ,g \rangle}
{H^{\alpha}},
\quad
:\hspace{-0.5mm}{\rm{sin}}(az)\hspace{-0.5mm}:(g)
=
\subscripts
 {H^{-\alpha}}
{\langle :\hspace{-0.5mm}{\rm{sin}}(az) \hspace{-0.5mm}: ,g \rangle}
{H^{\alpha}},
\quad \mu_{0}\mbox{-a.e. }z\in E.
\nonumber
\end{equation}
{\rm (3)}~If $\vert a \vert< {\sqrt{\frac{4\pi \alpha}{p-1}}}$, both
the Wick cosine $:\hspace{-0.5mm}{\rm{cos}}(a\cdot) \hspace{-0.5mm}:$
and the Wick sine $:\hspace{-0.5mm}{\rm{sin}}(a\cdot) \hspace{-0.5mm}:$
belong to $L^{p}(\mu_{0};H^{-\alpha}(\mathbb T^{2}))$.
Furthermore, if 
$\vert a \vert< {\sqrt{\frac{4\pi }{p-1}}}$,
both the functions
$:\hspace{-0.5mm}{\rm{cos}}(a\cdot) \hspace{-0.5mm}:(g)$ 
and
$:\hspace{-0.5mm}{\rm{sin}}(a\cdot) \hspace{-0.5mm}:(g)$ 
also belong to $L^{p}(\mu_{0})$.
\end{co}
%
%

Before closing this subsection, we present the exponential integrability
of the Wick trigonometric functions which plays a crucial role in the
proof of the main result. Although this property was first shown by Fr\"ohlich \cite{F, F2},
we provide a proof for the sake of completeness.
See also \cite{FP, AHk79}.
\begin{pr} 
\label{Zcos-positive}
For all $\vert a \vert <{\sqrt{4\pi}}$, the function
$\exp \big(-\hspace{-0.8mm}:\hspace{-0.8mm}{\rm{cos}}(a\cdot)\hspace{-0.8mm}: 
\hspace{-0.8mm}({\bf 1}_{\mathbb T^{2}}) \big )$
belongs to 
$L^{\infty-}(\mu_{0})$. Namely, for
any $p\geq 1$,
\begin{equation}
\int_{E} \exp 
\Big(-p \hspace{-0.8mm}:\hspace{-0.8mm}{\rm{cos}}(az)\hspace{-0.8mm}:
\hspace{-0.8mm}( {\bf 1}_{\mathbb T^{2}} ) \Big) 
\hspace{0.5mm}
\mu_{0}(dz)<\infty.
\label{SG-exp-int}
\end{equation}
Furthermore, we have
\begin{equation}
Z^{(a)}_{\sf cos}:=\int_{E} \exp 
\big(-\hspace{-0.8mm}:\hspace{-0.8mm}{\rm{cos}}(az)\hspace{-0.8mm}:
\hspace{-0.8mm}( {\bf 1}_{\mathbb T^{2}} ) \big) 
\hspace{0.5mm}
\mu_{0}(dz) >0.
\label{SG-exp-int2}
\end{equation}
If we replace 
$\hspace{-0.8mm}:\hspace{-0.8mm}{\rm{cos}}(a\cdot)\hspace{-0.8mm}: 
\hspace{-0.8mm}({\bf 1}_{\mathbb T^{2}})$ by
$\hspace{-0.8mm}:\hspace{-0.8mm}{\rm{sin}}(a\cdot)\hspace{-0.8mm}: 
\hspace{-0.8mm}({\bf 1}_{\mathbb T^{2}})$,
these properties still hold.
\end{pr}
\proof 
Since (\ref{SG-exp-int2}) can be shown in the same way as in the proof of
Proposition \ref{Z-positive}, we aim to prove (\ref{SG-exp-int}).
First, we fix $\Lambda \subset \mathbb R^{2}$ satisfying (\ref{Lambda-condition}) and
set $\Psi^{+}_{N}(z)=\big( {\rm{exp}}_{N\Lambda}({\sqrt{-1}}az), {\bf 1}_{\mathbb T^{2}} 
\big)_{H}$ and
$\Psi^{-}_{N}(z)=\big( {\rm{exp}}_{N\Lambda}(-{\sqrt{-1}}az), {\bf 1}_{\mathbb T^{2}} 
\big)_{H}$.
Noting ${\overline{\Psi^{+}_{N}(z)}}=\Psi^{-}_{N}(z)$ and $\big(\cos_{N\Lambda}(az), {\bf 1}_{\mathbb T^{2}}
\big)_{H}=\frac{1}{2} \big(  \Psi^{+}_{N}(z)+\Psi^{-}_{N}(z) \big)$, we have
\begin{align}
\int_{E}
\exp 
\Big \{-p \big( {\rm{cos}}_{N\Lambda}(az), {\bf 1}_{\mathbb T^{2}}  \big)_{H}\Big \} 
\hspace{0.5mm}
\mu_{0}(dz)
& \leq 2 \int_{E} {\rm cosh}
\Big \{p \big( {\rm{cos}}_{N\Lambda}(az), {\bf 1}_{\mathbb T^{2}}  \big)_{H}\Big \} \mu_{0}(dz)
\nonumber \\
&=2 \sum_{n=0}^{\infty} \frac{ p^{2n}}{(2n)!}
\int_{E} 
\big( {\rm{cos}}_{N\Lambda}(az), {\bf 1}_{\mathbb T^{2}}  \big)_{H}^{2n} \mu_{0}(dz)
\nonumber \\
&=2
 \sum_{n=0}^{\infty} \frac{ p^{2n}}{(2n)!}
\int_{E} 
\Big \{ \frac{1}{2}
\big(\Psi_{N}^{+}(z)+\Psi_{N}^{-}(z) \big) \Big \}^{2n} \mu_{0}(dz)
\nonumber \\
&\leq 
2
 \sum_{n=0}^{\infty} \frac{ p^{2n}}{(2n)!}
\int_{E} 
\frac{1}{2^{2n}}
\sum_{m=0}^{2n}
\binom{2n}{m}
\vert \Psi_{N}^{+}(z) \vert^{m} \vert \Psi_{N}^{-}(z) \vert^{2n-m}  \mu_{0}(dz)
\nonumber \\
&\leq 
2
 \sum_{n=0}^{\infty} \frac{ p^{2n}}{(n!)^{2}}
\int_{E} 
\big(\Psi_{N}^{+}(z) \big)^{n}\big(\Psi_{N}^{-}(z) \big)^{n} \mu_{0}(dz),
\label{Psi-estimate}
\end{align}
where we used 
$$\vert \Psi^{+}_{N}(z) \vert^{m} 
\vert \Psi^{-}_{N}(z) \vert^{2n-m}=\vert \Psi^{+}_{N}(z) \vert^{2n}= 
\big( \Psi^{+}_{N}(z) 
\Psi^{-}_{N}(z) \big)^{n}, \quad m=0,\ldots, 2n$$
and the identities $\sum_{m=0}^{2n} \binom{2n}{m}=2^{2n}$ and $(2n)!\geq (n!)^{2}$ for 
the final line in the above estimate.

For given $i\in \{1,\ldots, 2n \}$, we set $\varepsilon_{i}=1$ if $i$ is even and
$\varepsilon_{i}=-1$ if $i$ is odd.
Then 
\begin{align}
\int_{E}
&
\big(\Psi_{N}^{+}(z) \big)^{n}\big(\Psi_{N}^{-}(z) \big)^{n} \mu_{0}(dz)
\nonumber \\
&
=\exp \big( a^{2} \rho_{N\Lambda}^{2} \big) \int_{(\mathbb T^{2})^{2n}}dx_{1} \cdots 
dx_{2n} 
\int_{E} \exp \Big \{ {\sqrt{-1}}a \sum_{i=1}^{n}
\Big (
(\Pi_{N\Lambda}z)(x_{2i})-
(\Pi_{N\Lambda}z)(x_{2i-1})
\Big)
\Big \} \mu_{0}(dz)
\nonumber \\
&= \exp \big( a^{2} \rho_{N\Lambda}^{2} \big) 
\int_{(\mathbb T^{2})^{2n}}
\exp \Big(-\frac{(a\rho_{N\Lambda})^{2}}{2} \Big \Vert \sum_{i=1}^{n} \big(
\eta_{N\Lambda, x_{2i}}
-
\eta_{N\Lambda, x_{2i-1}}
\big)
\Big \Vert_{H^{-1}}^{2} \Big) 
dx_{1} \cdots 
dx_{2n} 
\nonumber \\
&=
\int_{(\mathbb T^{2})^{2n}}
\exp \Big ( -a^{2}\sum_{1\leq i <j \leq 2n} \varepsilon_{i} \varepsilon_{j} K^{(1)}_{N\Lambda} (x_{i}-x_{j})
\Big )
dx_{1} \cdots 
dx_{2n} 
\nonumber \\
&\leq 
\int_{(\mathbb T^{2})^{2n}}
\exp \Big ( -a^{2}\sum_{1\leq i <j \leq 2n} \varepsilon_{i} \varepsilon_{j} K^{(1)}(x_{i}-x_{j})
\Big )
dx_{1} \cdots 
dx_{2n}, 
\label{stop-estimate}
\end{align}
where we used that the Green functions $K^{(1)}_{N\Lambda}$ and $K^{(1)}$ satisfies
$K^{(1)}_{N\Lambda} \geq K^{(1)}$ as quadratic forms on $\mathbb T^{2}$ for the final line.
(Note that the final line can be regarded as the {\it{conditioning property}} mentioned in \cite[Lemma IV.3]{F2}.)
Recalling (\ref{kasaneawase}) and \cite[(4.2)]{AS61} (see also Proposition \ref{Green} below), we obtain
\begin{equation}
K^{(1)}(x)=-\frac{1}{2\pi} \log \vert x \vert_{\mathbb R^{2}}+R(x), \quad x\in {\mathbb T}^{2} \setminus \{0 \},
\label{AS42-61}
\end{equation}
for some smooth function $R$.
Combining (\ref{AS42-61}) with
\cite[Lemma 2.1]{F}, we also have
\begin{align}
\int_{(\mathbb T^{2})^{2n}}
&
\exp \Big ( -a^{2}\sum_{1\leq i <j \leq 2n} \varepsilon_{i} \varepsilon_{j} K^{(1)}(x_{i}-x_{j})
\Big )
dx_{1} \cdots 
dx_{2n} \nonumber \\
&\leq
\exp({2a^{2}Cn}) 
\int_{(\mathbb T^{2})^{2n}}
\prod_{1\leq i <j \leq 2n} \vert x_{i}-x_{j} \vert_{\mathbb R^{2}}^{\varepsilon_{i} \varepsilon_{j} a^{2}/2\pi}
dx_{1} \cdots 
dx_{2n}. 
\label{dipole}
\end{align}
Furthermore, it follows from \cite[(1.5), (2.7) and (2.8)]{DL74} that
\begin{align}
\int_{(\mathbb T^{2})^{2n}}
&
\prod_{1\leq i <j \leq 2n} \vert x_{i}-x_{j} \vert_{\mathbb R^{2}}^{\varepsilon_{i} \varepsilon_{j} a^{2}/2\pi}
dx_{1} \cdots 
dx_{2n}
\nonumber \\
&\leq 
(4\pi^{2})^{2n+\frac{a^{2}}{4\pi}\sum_{1\leq i<j\leq 2n}\varepsilon_{i}\varepsilon_{j}}
(n!)^{\max \{a^{2}/2\pi,  1\}}(Q_{1}^{*})^{n}=C^{n}
(n!)^{\max \{a^{2}/2\pi,  1\}},
\label{DS74}
\end{align}
where the positive constant $Q_{1}^{*}$ is given by
\begin{equation}
Q_{1}^{*}=\frac{\pi^{2} \Gamma(2-a^{2}/2\pi)}{(1-a^{2}/8\pi)(1-a^{2}/4\pi)\Gamma (1-
a^{2}/4\pi)\Gamma (2-a^{2}/4\pi)}. 
\nonumber
\end{equation}

Inserting (\ref{stop-estimate}), (\ref{dipole}) and (\ref{DS74}) into (\ref{Psi-estimate}), we obtain
\begin{equation}
\sup_{N\in \mathbb N}
\int_{E}
\exp 
\Big (-p 
\big( {\rm{cos}}_{N\Lambda}(az), {\bf 1}_{\mathbb T^{2}}  \big)_{H}
\Big ) 
\hspace{0.5mm}
\mu_{0}(dz)
\leq
\sum_{n=0}^{\infty}\frac{C^{n}}{(n!)^{\min \{2-a^{2}/2\pi, 1\}}} (p^{2})^{n}.
\label{final-sG-estimate}
\end{equation}
Noting that the condition $\vert a \vert <4\pi$ implies $2-a^{2}/2\pi>0$,
we see that the right-hand side of (\ref{final-sG-estimate}) converges absolutely for every $p\geq 1$.
This estimate implies that $\big \{ \exp  
\big (-p 
\big( {\rm{cos}}_{N\Lambda}(a\cdot), {\bf 1}_{\mathbb T^{2}}  \big)_{H}
\big )  
\big \}_{N\in \mathbb N}$ is uniformly integrable.
Hence by combining this property with Corollary \ref{Wick-Tri},
we finally obtain the desired estimate (\ref{SG-exp-int}).
\qed
\section{Proof of main results}
In this section, we 
only consider the case $\mu=\mu^{(a)}_{\cos}$ and 
give a proof of Proposition \ref{IbP} and Theorem \ref{ES}.
Since we have
$\exp (-\hspace{-0.8mm}:\hspace{-0.8mm} \exp (a\cdot) \hspace{-0.8mm}:\hspace{-0.5mm} ({\bf 1}_{\mathbb T^{2}})) 
\in L^{\infty}(\mu_{0})$ (see Proposition \ref{Z-positive}), which is a stronger property than 
$\exp (-:\hspace{-0.8mm} \cos (a\cdot) \hspace{-0.8mm}:\hspace{-0.5mm} ({\bf 1}_{\mathbb T^{2}})) 
\in L^{\infty-}(\mu_{0})$ (see Proposition \ref{Zcos-positive}),
the proof in the case of 
$\mu=\mu^{(a)}_{\exp}$
goes in a very similar way with only slight modifications.
\noindent
\begin{proof}[{\bfseries Proof of Proposition \ref{IbP}}]
Applying a standard Gaussian integration by parts formula (cf. \cite[page 207]{GJ}), we easily have
\begin{align}
\int_{E}
& \big(D_{H}F(z),\varphi \big)_{H} 
\frac{1}{Z^{(a)}_N} 
\exp \Big(-\big(\cos_{D(N)}(az), {\bf 1}_{\mathbb T^{2}} \big)_{H} \Big)
\hspace{0.5mm}
\mu_{0}(dz)
\nonumber \\
&=
\int_{E}F(z)\Big \{ \langle z, (1-\Delta)\varphi \rangle
+a \big(\sin_{D(N)}(az), \varphi \big)_{H}
\Big \}
\frac{1}{Z^{(a)}_N} 
\exp \Big(-\big(\cos_{D(N)}(az), {\bf 1}_{\mathbb T^{2}} \big)_{H} \Big)
\hspace{0.5mm}
\mu_{0}(dz)
\label{IbP-beta}
\end{align}
for all $F\in {\mathfrak F}C_{b}^{\infty}, \varphi \in K$ and $N\in \mathbb N$, where
$$ 
Z^{(a)}_N:=\int_{E} \exp \Big(-\big(\cos_{D(N)}(az), {\bf 1}_{\mathbb T^{2}} \big)_{H} \Big)
\mu_{0}(dz).
$$
Combining Corollary \ref{Wick-Tri} with 
Proposition \ref{Zcos-positive} (in particular,
(\ref{final-sG-estimate})),
we can take a
subsequence $\{N(j) \}_{j=1}^{\infty}$ with $N(j) \nearrow \infty$ as $j \to \infty$ such that
$$\lim_{j\to \infty} \big(\cos_{D(N)}(az), \varphi \big)_{H}=
\hspace{0.8mm}:\hspace{-0.8mm} \cos(az) \hspace{-0.8mm}:\hspace{-0.5mm}(\varphi),
\quad 
\lim_{j\to \infty} \big(\sin_{D(N)}(az), \varphi \big)_{H}=
\hspace{0.8mm}:\hspace{-0.8mm} \sin(az) \hspace{-0.8mm}:\hspace{-0.5mm}(\varphi),
\quad \mu_{0}\mbox{-a.s. }z\in E,
$$
and $\lim_{j\to \infty}Z_{N}^{(a)}=Z_{\sf cos}^{(a)}$ provided that $\vert a \vert <{\sqrt{4\pi}}$.
Thus, taking the limit $j \to \infty$ on both sides of 
(\ref{IbP-beta}), we obtain
$$
\int_{E}
 \big(D_{H}F(z),\varphi \big)_{H} 
 \hspace{0.5mm}
\mu^{(a)}_{\sf cos}(dz)
\nonumber \\
=
\int_{E}F(z)\Big \{ \langle z, (1-\Delta)\varphi \rangle
+a :\hspace{-0.8mm} \sin(az) \hspace{-0.8mm}:\hspace{-0.5mm}(\varphi) 
\Big \}
\hspace{0.5mm}
\mu^{(a)}_{\sf cos}(dz),
$$
and this leads us to the desired integration by parts formula (\ref{IbP3}).

To show ${\mathcal L}F\in L^{p}(\mu^{(a)}_{\sf cos})$, it is sufficient to check
$:\hspace{-0.8mm} \sin (az) \hspace{-0.8mm}: (\varphi) \hspace{0.5mm} \in L^{p}(\mu^{(a)}_{\sf cos})$.
Applying Corollary \ref{Wick-Tri} and 
Proposition \ref{Zcos-positive} again, we easily see
$:\hspace{-0.8mm} \sin (az) \hspace{-0.8mm}: \hspace{-0.5mm}(\varphi)
 \in L^{p}(\mu^{(a)}_{\sf cos})$,
provided that $\vert a \vert< {\sqrt{\frac{4\pi }{p-1}}}$. This completes the proof.
\end{proof}
\begin{proof}[{\bfseries Proof of  Theorem \ref{ES}}]
First, we introduce 
${\mathcal H}_{+}:=H^{2\gamma+\delta}({\mathbb T^{2}})$ and ${\mathcal H}_{-}:=E=H^{-\delta}(\mathbb T^{2})$.
By virtue of condition (\ref{index-assume}), the embeddings 
${\mathcal H}_{+} \subset {\mathcal H} \subset {\mathcal H}_{-}$ are 
Hilbert-Schmidt. Note that ${\mathcal H}_{+}$ is regarded as the dual space of ${\mathcal H}_{-}$ if 
we identify ${\mathcal H}$ with its dual.
Let $\mbox{ }_{-}(\cdot, \cdot)_{+}$ stand for this dualization between ${\mathcal H}_{-}$ 
and ${\mathcal H}_{+}$. Noting
the identity
$$ \mbox{ }_{-}(z,\varphi)_{+}=(z, \varphi)_{\mathcal H}, \quad \varphi \in {\mathcal H}_{+}, z\in {\mathcal H},$$
we may rewrite ${\mathcal L}F$ given in Proposition \ref{IbP} as
\begin{equation}
{\mathcal L}F(z)=\frac{1}{2}\big \{ {\rm Tr}(D_{\mathcal H}^{2}F(z))+ 
\hspace{-2mm}\mbox{ }_{-}\big ( \beta(z), D_{\mathcal H}F(z) \big)_{+} \big \}, \quad F\in {\mathfrak F}C^{\infty}_{b},
\nonumber
\end{equation}
where $\beta(z):=\beta_{\sf{OU}}(z)+\beta_{\sf{cos}}^{(a)}(z)$ is given by
\begin{eqnarray}
\beta_{\sf{OU}}(z)&=&-\sum_{k\in \mathbb Z^{2}} \lambda_{k}^{-\delta/2-\gamma+1} 
\langle z,e_{k} \rangle e_{k}^{(-\delta)}
\nonumber
\\
\beta_{\sf{cos}}^{(a)}(z)&=&a \sum_{k\in \mathbb Z^{2}} \lambda_{k}^{-\gamma/2} \sum_{n=0}^{\infty}
(-1)^{n}
 \frac{a^{2n+1}}{(2n+1)!}
\Phi_{2n+1}(z)(e_{k}) \hspace{0.5mm} e_{k}^{(\gamma)}
\nonumber
\end{eqnarray}

Thanks to \cite[Theorem 3]{LR}, it is sufficient to check 
$\beta_{\sf{OU}}\in L^{2p}(\mu^{(a)}_{\sf cos};{\mathcal H}_{-})$ and 
$\beta_{\sf{cos}}^{(a)} \in L^{2p}(\mu^{(a)}_{\sf cos}; {\mathcal H})$ for 
$L^{p}$-uniqueness of the Dirichlet operator ${\mathcal L}$. 
By using the H\"older inequality and remembering the condition
$\delta+2\gamma>2$, we can estimate as follows:
\begin{align} \int_{E} &
\big \Vert \beta_{\sf{OU}}(z) \big \Vert_{E}^{2p} \mu^{(a)}_{\sf cos}(dz)
\nonumber \\
& \leq
{\rm{Tr}}(A^{\delta+2\gamma-1})^{p-1}
\int_{E} 
\sum_{k\in {\mathbb Z}^{2}} \lambda_{k}^{-(\delta+2\gamma-p-1)} \langle z, e_{k} \rangle^{2p}  
\mu^{(a)}_{\sf cos}(dz)
\nonumber \\
& \lesssim 
{\rm{Tr}}(A^{\delta+2\gamma-1})^{p-1}
\sum_{k\in {\mathbb Z}^{2}} \lambda_{k}^{-(\delta+2\gamma-p-1)} \Big( 
\int_{E}  \langle z, e_{k} \rangle^{2p}  
\exp \big( -:\hspace{-0.8mm} \cos (az) \hspace{-0.8mm}:\hspace{-0.5mm} ({\bf 1}_{\mathbb T^{2}}) \big) 
\hspace{0.5mm} \mu_{0}(dz) \Big)
\nonumber \\
& =
{\rm{Tr}}(A^{\delta+2\gamma-1})^{p-1}
\sum_{k\in {\mathbb Z}^{2}} \lambda_{k}^{-(\delta+2\gamma-1)} 
\Big ( \int_{E}  \langle z, 
e^{(-1)}_{k}
 \rangle^{2p}  
\exp \big( -:\hspace{-0.8mm} \cos (az) \hspace{-0.8mm}:\hspace{-0.5mm} ({\bf 1}_{\mathbb T^{2}}) \big) 
\hspace{0.5mm} 
\mu_{0}(dz) \Big )
\nonumber \\
&= {\rm{Tr}}(A^{\delta+2\gamma-1})^{p} \Big
\{ \int_{\mathbb R} x^{4p} \big( \frac{1}{{\sqrt{2\pi}}} e^{-\frac{x^{2}}{2}} dx \big)
\Big \}^{1/2}
\Big (
\int_{E} \exp \big( -2:\hspace{-0.8mm} \cos (az) \hspace{-0.8mm}:\hspace{-0.5mm} ({\bf 1}_{\mathbb T^{2}}) \big) 
\hspace{0.5mm} 
\mu_{0}(dz)
\Big)^{1/2}
<\infty,
\nonumber
\end{align}
where we used 
$ \langle z, e^{(-1)}_{k} \rangle \sim N(0,1)$ ($k\in \mathbb Z^{2}$) for the final line.

On the other hand, by using the triangle inequality and recalling Proposition
\ref{Zcos-positive},
we obtain
\begin{align}
\Big( \int_{E} 
\big \Vert \beta_{\sf cos}^{(a)}(z) \big \Vert_{\mathcal H}^{2p} \mu^{(a)}_{\sf cos}(dz) \Big)^{1/2p} 
&\leq  \vert a \vert \sum_{n=0}^{\infty} \frac{ \vert a \vert^{2n+1}}{(2n+1)!} 
\Big [ \int_{E} \Big \{ \sum_{k\in \mathbb Z^{2}} \Big( \lambda_{k}^{-\gamma/2} 
\Phi_{2n+1}(z)(e_{k}) \Big)^{2} \Big \}^{p} \mu^{(a)}_{\sf cos}(dz) \Big ]^{1/2p}
\nonumber \\
& \lesssim  \vert a \vert \sum_{n=0}^{\infty} \frac{ \vert a \vert^{2n+1}}{(2n+1)!} 
\Big( \int_{E} \big \Vert  \Phi_{2n+1}   \big \Vert_{H^{-\gamma}}^{2p(1+\varepsilon)} \hspace{0.5mm}
\mu_{0}(dz)
\Big )^{1/2p(1+\varepsilon)}
\nonumber \\
&\mbox{ } \hspace{15mm} \times
\Big( \int_{E} 
\exp \big( -\frac{1+\varepsilon}{\varepsilon}
:\hspace{-0.8mm} \cos (az) \hspace{-0.8mm}:\hspace{-0.5mm} ({\bf 1}_{\mathbb T^{2}}) \big) 
\hspace{0.5mm} 
\mu_{0}(dz)
\Big )^{\varepsilon/2p(1+\varepsilon)}
\nonumber \\
&\lesssim
\vert a \vert \sum_{n=0}^{\infty} \frac{ \vert a \vert^{2n+1}}{(2n+1)!} \big \Vert \Phi_{2n+1} \big 
\Vert_{L^{2p(1+\varepsilon)}(\mu_{0};H^{-\gamma}(\mathbb T^{2}))}
\label{funano}
\end{align}
for any $\varepsilon>0$.
By virtue of 
Corollary \ref{Wick-Tri}, the right-hand side of (\ref{funano}) is finite under the condition $2p(1+\varepsilon) 
<1+\frac{4\pi\gamma}{a^{2}}$.
Since we may take $\varepsilon>0$ sufficiently small,
this completes the proof.
\end{proof}
%
\renewcommand{\theequation}{A.\arabic{equation}}
\renewcommand{\thesection}{A}
\setcounter{tm}{0}
\setcounter{equation}{0}
%
\section{Auxiliary 
results}
In this appendix, 
for the reader's convenience, 
we collect
auxiliary results
which are used in the proof of
the main results.
\begin{lm}
\label{HY-Z}
Let $\Lambda$ and $\Lambda'$ be subsets of $\mathbb R^{2}$ satisfying {\rm{(\ref{Lambda-condition})}}, and 
$K^{(1)}_{\Lambda}$
be the approximating Green kernel of $(1-\Delta)$ defined in {\rm{(\ref{KN-delta})}}. 
Then for any $n\in \mathbb N$, $q\geq 2$, 
we have
\begin{eqnarray}
\Vert K_{\Lambda}^{(1)} \Vert_{L^{qn}(\mathbb T^{2})}
&\leq & \big( \frac{1}{{2}\pi} \big)^{\frac{2(qn-1)}{qn}}
\Big \{ \sum_{k\in \mathbb Z^{2}} {\bf 1}_{\Lambda}(k)
\big(\frac{1}{1+\vert k \vert^{2}} \big)^{\frac{qn}{qn-1}} \Big \}^{\frac{qn-1}{qn}},
\label{HY1}
\\
\Vert K_{\Lambda'}^{(1)}-K_{\Lambda}^{(1)} \Vert_{L^{qn}(\mathbb T^{2})}
&\leq & \big( \frac{1}{{2}\pi} \big)^{\frac{2(qn-1)}{qn}}
\Big \{ \sum_{k\in \mathbb Z^{2}} 
{\bf 1}_{\Lambda' \setminus \Lambda}(k) 
\big(\frac{1}{1+\vert k \vert^{2}} \big)^{\frac{qn}{qn-1}} \Big \}^{\frac{qn-1}{qn}}, \quad \Lambda \subset \Lambda'.
\label{HY2}
\end{eqnarray}
\end{lm}
\begin{proof}
By recalling the Hausdorff-Young inequality (cf. \cite[Theorem 2.8 in Chapter XII]{Zyg}), we have
%
\begin{equation}
\Vert f \Vert_{L^{r}({\mathbb T}^{2})} \leq \big (\frac{1}{2\pi} \big)^{\frac{2}{s}-1} 
\Big (
\sum_{k\in \mathbb Z^{2}}
\vert {\hat f}(k)\vert^{s}
\Big)^{1/s},\qquad f\in L^{r}({\mathbb T}^{2}), \quad r\geq 2,
\label{HY}
\end{equation}
where 
$1\leq s\leq 2$ 
is the conjugate index of $r$.
Now we set $f=K^{(1)}_{\Lambda}$. Replacing $r$ and $s$ by
$qn(\geq 2)$ and $\frac{qn}{qn-1}$, respectively, we easily see that
(\ref{HY}) implies estimate (\ref{HY1}). Similarly, estimate (\ref{HY2}) is 
obtained by putting $f=K_{\Lambda'}-K_{\Lambda}$.
\end{proof}
\begin{pr} \label{Green}
Let
$G^{(\alpha)}(x,y)=G^{(\alpha)}(x-y)$ 
$(0<\alpha \leq 1, x,y\in {\mathbb R}^{2})$ be the Green function for
$(1-\Delta)^{\alpha}$ on $\mathbb R^{2}$. 
Then $G^{(\alpha)}$ is smooth on $\mathbb R^{2} \setminus \{0 \}$ and strictly positive.
Furthermore, there exist some constants $C_{1}, C_{2}, C_{3}>0$ such that the following holds:
if $0<\alpha<1$, then 
\begin{equation}
G^{(\alpha)}(x-y)
\leq
\begin{cases}
\displaystyle{C_{1} \vert x-y \vert_{\mathbb R^{2}}^{2\alpha-2}}
& \text{ for $\vert x-y \vert_{\mathbb R^{2}} <1$}
\vspace{1.5mm} \\ 
\displaystyle{
C_{2} \exp(-C_{3}\vert x-y \vert_{\mathbb R^{2}})
}
& \text{ for $\vert x-y \vert_{\mathbb R^{2}} \geq 1$},
\end{cases}
\label{A10}
\end{equation}
and if $\alpha=1$, then
\begin{equation}
G^{(1)}(x-y)
\leq
\begin{cases}
\displaystyle{
-\frac{1}{2\pi} \log  \vert x-y \vert_{\mathbb R^{2}}+C_{1}}
& \text{ for $\vert x-y \vert_{\mathbb R^{2}} <1$}
\vspace{1.5mm} \\ 
\displaystyle{
C_{2} \exp(-C_{3}\vert x-y \vert_{\mathbb R^{2}})
}
& \text{ for $\vert x-y \vert_{\mathbb R^{2}} \geq 1$}.
\end{cases}
\label{A11}
\end{equation}
\end{pr}
\begin{proof}
One can find a proof of (\ref{A11}) e.g., in
\cite[Section 4]{AS61},
\cite[Proposition V.23]{Si} and 
\cite[Proposition 7.2.1]{GJ}.
Although the case $0<\alpha<1$ is also treated in 
\cite[Section 4]{AS61},
we give a short proof of (\ref{A10}) here
to make the present paper more self-contained.

We first recall an integral representation of the Green function 
$$ G^{(\alpha)}(x-y)=\frac{1}{4\pi \Gamma(\alpha)} \int_{0}^{\infty} e^{-s} \exp 
\big( -\frac{ \vert x-y \vert^{2}_{\mathbb R^{2}}}{4s}
\big) s^{\alpha-2} ds, \qquad 0<\alpha \leq 1,~ x,y \in \mathbb R^{2}.
$$
Setting $s=\frac{ \vert x-y \vert_{\mathbb R^{2}} 
}{t}$, we have 
\begin{equation}
G^{(\alpha)}(x-y)=\frac{
\vert x-y \vert_{\mathbb R^{2}}^{2\alpha-2}}{4\pi \Gamma(\alpha)} I(
\vert x-y \vert_{\mathbb R^{2}}),
\label{A12}
\end{equation}
where
$$
I(r)=\int_{0}^{\infty}
t^{-\alpha} \exp \big(-\frac{t}{4}-\frac{r^{2}}{t} \big) dt, \qquad r\geq 0.
$$
Combining (\ref{A12}) with 
$$ I(r)\leq \int_{0}^{\infty} t^{-\alpha} e^{-t/4} dt =4^{1-\alpha}\Gamma(1-\alpha),$$
we obtain the desired estimate in the case 
$\vert x-y \vert_{\mathbb R^{2}}<1$.

Next we consider the case $r=\vert x-y \vert_{\mathbb R^{2}}\geq 1$. We set
${\tau}=\frac{\sqrt t}{2}-\frac{r}{\sqrt t}$, in other word,
$t=({\tau}+{\sqrt{
{\tau}^{2}+2r}})^{2}$. Then 
we have
\begin{align}
I(r)&=2e^{-r} \int_{\mathbb R} 
\frac{\big({\tau}+{\sqrt{
{\tau}^{2}+2r}} \big)^{2-2\alpha}
}
{
{\sqrt{
{\tau}^{2}+2r}}
}
e^{-\tau^{2}}d\tau
\nonumber \\
&= {2}e^{-r}r^{\frac{1-2\alpha}{2}} \int_{\mathbb R}
\frac{ 
\Big(\frac{\tau}{\sqrt r}+{\sqrt{
\frac{{\tau}^{2}}{r}+2}} \Big)^{2-2\alpha}
}
{\sqrt{ \frac{\tau^{2}}{r}+2}
}
e^{-\tau^{2}} d\tau
\nonumber \\
&\leq 
{\sqrt 2}
e^{-r}r^{\frac{1-2\alpha}{2}} \int_{\mathbb R}
\big({\tau}+{\sqrt{
{\tau}^{2}+2}} \big)^{2-2\alpha}
e^{-\tau^{2}} d\tau
~
\lesssim ~e^{-r}r^{\frac{1-2\alpha}{2}},
\nonumber
\end{align}
where we used $r\geq 1$ and $2-2\alpha>0$ for the third line.
Combining this estimate with (\ref{A12}), we finally obtain
$$ G^{(\alpha)} (x-y) \leq C \exp (-\vert x-y \vert_{\mathbb R^{2}})
\vert x-y \vert_{\mathbb R^{2}}^{\alpha-\frac{3}{2}} \leq 
C \exp (-\vert x-y \vert_{\mathbb R^{2}}),
$$
where we used 
$\vert x-y \vert_{\mathbb R^{2}}^{\alpha-\frac{3}{2}}\leq 1$. This
completes the proof.
\end{proof}
\noindent
{\textbf{Acknowledgments.}}~The authors are grateful to Professor Seiichiro Kusuoka
for his valuable suggestions and constant encouragement during the preparation of the
present paper. 
H.K. also thanks Professors Masato Hoshino,
Tomoyuki Kakehi and Kaneharu Tsuchida
for helpful discussions. 
This work was initiated during a stay of 
H.K. at Hausdorff Center for Mathematics (HCM), 
Universit\"at Bonn in fall 2012. 
He thanks Professor Benjamin Schlein (now at
Universit\"at Z\"urich) 
for a kind invitation and the financial support 
during his stay at HCM.
He was also partially supported by JSPS through the Grant-in-Aid for Scientific Research (C) 
26400134 and 17K05300.
H.K. and M.R. were partially supported by DFG through
CRC 701 ``Spectral Structures and Topological Methods in Mathematics" and
CRC 1283
``Taming uncertainty and profiting from randomness and low regularity in analysis, 
stochastics and their applications".

\end{document}